\documentclass[11pt,a4paper]{amsart}
\usepackage[english]{babel}
\usepackage[utf8]{inputenc}
\usepackage[T1]{fontenc}
\usepackage{csquotes} 
\usepackage{amssymb}
\usepackage{float}
\usepackage{amsthm}
\usepackage{mathtools} 
 \usepackage[top=1.7cm, bottom=1.7cm, left=1.7cm, right=1.7cm,twoside=false]{geometry}
\usepackage{color}
\usepackage[usenames,x11names]{xcolor}
\usepackage[all]{xy}
\usepackage{enumitem}
\setlist[1]{leftmargin=*}
\setlist[enumerate,1]{label=(\alph*)}
\setlist[enumerate,2]{label=\normalfont{(\roman*)}, ref=\normalfont{(\alph{enumi}.\roman*)}}

\usepackage{multirow}
\usepackage{caption}
\usepackage{subcaption}
\usepackage[export]{adjustbox}
\usepackage[hypertexnames=false]{hyperref} 
\hypersetup{linkcolor  =DodgerBlue3, citecolor  = teal, urlcolor   = teal, colorlinks = true}

\usepackage[normalem]{ulem}

\newcommand{\A}{\mathbb{A}}
\newcommand{\C}{\mathbb{C}}
\newcommand{\D}{\mathbb{D}}

\renewcommand{\P}{\mathbb{P}}

\newcommand{\R}{\mathbb{R}}
\newcommand{\Z}{\mathbb{Z}}

\newcommand{\cA}{\mathcal{A}}

\newcommand{\cF}{\mathcal{F}}

\newcommand{\cH}{\mathcal{H}}

\newcommand{\cM}{\mathcal{M}}

\newcommand{\cO}{\mathcal{O}}

\usepackage[scr=boondoxo]{mathalfa}
\renewcommand{\ll}{\mathscr{l}}
\newcommand{\kk}{\mathscr{k}}

\renewcommand{\epsilon}{\varepsilon}
\renewcommand{\phi}{\varphi}
\renewcommand{\theta}{\vartheta}

\renewcommand{\tilde}{\widetilde}
\renewcommand{\hat}{\widehat}
\renewcommand{\bar}{\overline}

\renewcommand{\leq}{\leqslant}
\renewcommand{\geq}{\geqslant}

\renewcommand{\to}{\longrightarrow}

\newcommand{\into}{\hookrightarrow}

\newcommand{\de}{\coloneqq} 

\newcommand{\Spec}{\operatorname{Spec}}

\newcommand{\Aut}{\operatorname{Aut}}
\newcommand{\Autalg}{\mathrm{Aut}_{\mathrm{alg}}}
\newcommand{\Authol}{\mathrm{Aut}_{\mathrm{hol}}}
\newcommand{\End}{\operatorname{End}}
\newcommand{\Aff}{\operatorname{Aff}}
\newcommand{\Int}{\operatorname{Int}} 

\newcommand{\sspan}{\operatorname{span}}

\newcommand{\Nn}{\Z_{\geq 0}}

\renewcommand{\d}{\partial}
\newcommand{\id}{\mathrm{id}}

\newcommand{\Gl}{\mathrm{GL}}

\theoremstyle{plain}
\newtheorem{thm}{Theorem}[section]
\newtheorem{cor}[thm]{Corollary}

\theoremstyle{definition}
\newtheorem{dfn}[thm]{Definition}
\newtheorem{lem}[thm]{Lemma}
\newtheorem{prop}[thm]{Proposition}
\newtheorem{ex}[thm]{Example}

\newtheorem{rem}[thm]{Remark}

\theoremstyle{remark}
\newtheorem{claim}{Claim}
\newtheorem*{claim*}{Claim}
\newtheorem*{rem*}{Remark}

\numberwithin{equation}{section}

\makeatletter \def\subsection{\@startsection{subsection}{3}
	\z@{.5\linespacing\@plus.7\linespacing}{.5\linespacing}
	{\bfseries\itshape}} \makeatother

\makeatletter \renewenvironment{proof}[1][\proofname]{
	\par\pushQED{\qed}\normalfont
	\topsep6\p@\@plus6\p@\relax
	\trivlist\item[\hskip\labelsep\bfseries#1\@addpunct{.}]
	\ignorespaces}{
	\popQED\endtrivlist\@endpefalse} \makeatother
\begin{document}
	
\title[Linearization of holomorphic families of algebraic automorphisms of $\C^{2}$] 
{Linearization of holomorphic families of algebraic automorphisms of the affine plane}

\author{Shigeru Kuroda}
\address{Department of Mathematical Sciences\\
	Tokyo Metropolitan University\\
	1-1 Minami-Osawa, Hachioji, Tokyo, 192-0397, Japan}
\email{kuroda@tmu.ac.jp}
\thanks{The first author was partially supported by JSPS KAKENHI Grant Number 18K03219}

\author{Frank Kutzschebauch}
\address{Departement Mathematik\\
	Universit\"at Bern\\
	Sidlerstrasse 5, CH--3012 Bern, Switzerland}
\email{frank.kutzschebauch@math.unibe.ch}
\thanks{The second author was partially supported by Schweizerischer Nationalfonds Grant 200021-116165 }

\author{Tomasz Pe\l ka}
	\address{Basque Center for Applied Mathematics, Alameda de Mazarredo 14, 48009 Bilbao, Spain}
\email{tpelka@bcamath.org}

\keywords{linearization of group actions, Oka properties, affine plane}
\subjclass[2010]{Primary: 14R20,  32M05; Secondary: 14R10, 32M17, 32Q56}
\bibliographystyle{amsalpha}
\begin{abstract} 
	Let $G$ be a reductive group. We prove that  a family  of polynomial actions of $G$ on $\C^2$, holomorphically parametrized by an open Riemann surface, is linearizable. As an application, we show that a particular class of reductive group actions on $\C^3$ is linearizable. The main step of our proof is to establish a certain restrictive  Oka property  for groups  of equivariant algebraic automorphisms of $\C^2$.
\end{abstract}

\maketitle


\section{Introduction}

Let $G$ be a reductive subgroup of the algebraic (or holomorphic) automorphism group of $\C^n$. The \emph{Linearization Problem} asks if there are algebraic (or holomorphic) coordinates on $\C^n$ in which $G$ acts linearly. If this is the case, we say that $G$ is \emph{algebraically} (or \emph{holomorphically}) \emph{linearizable}. Both algebraic and holomorphic versions of Linearization Problem gained much attention since 1970s, and their complete solutions still remain elusive: for a survey, see e.g.\ \cite{Kraft-challenging_problems,KS_reductive,Ku,DK} and references therein.
\smallskip

For a general $G$ and $n$, the answer to this problem is negative. In the algebraic case, for any non-abelian, semi-simple $G$ there are counterexamples with $n\geq 4$ \cite{Knop}; first ones constructed by Schwarz \cite{Schwarz89}. They all come from non-trivial $G$-vector bundles over representations. However, if $G$ is abelian,  all such $G$-bundles are trivial \cite{Masuda}, so a method to construct counterexamples is yet to be found. 

A positive answer to the algebraic Linearization Problem is known for $n\leq 2$ \cite{Kambayashi_linearization} and for $n=3$ with $G$ infinite \cite{KR_families-of-group-actions}. The former follows from a classical structure theorem \cite{van_der_Kulk}, cf.\ Lemma \ref{lem:generic_linearization}, while the latter relies on a deep result \cite{KKMLR_Cst-linearizable} settling the case $G=\C^{*}$. For $n=3$ and $G$ finite, the algebraic Linearization Problem is open. But when the $G$-action fixes a variable, a positive answer follows from  \cite[Theorem D]{KR_families-of-group-actions}. Our Theorem \ref{thm:main} and Corollary \ref{cor:C3} are holomorphic counterparts of this result.
\smallskip

The holomorphic Linearization Problem is, in a way, even more subtle than the algebraic one. In case $n=2$, it has a positive answer for $G=\C^{*}$ \cite{Suzuki_C2}, but for finite $G$ it is widely open \cite[Remark 4.7]{DK}, see Example \ref{ex:C2-nonlinearizable}. Moreover, all $G$-bundles over representations are holomorphically trivial \cite{HK}, so the algebraic counterexamples do not generalize to this setting. Nonetheless, Derksen and the second author \cite{DK} found counterexamples for any $G$ and any $n\geq N_{G}$, where $N_{G}$ is a constant depending on $G$. As in the algebraic case, the smallest dimension where a counterexample is known is $n=4$, for $G=\Z_2$. 
\smallskip

In this article, we explore what happens between these two categories. Namely, we study algebraic $G$-actions which vary in holomorphic families. We make it precise in Definition \ref{def:family} below. 

Throughout the article, we denote by $\Autalg(\C^n)$ and $\Authol(\C^n)$ the groups of algebraic and holomorphic automorphisms of $\C^n$, respectively; and treat $\Autalg(\C^n)$ as a subgroup of $\Authol(\C^n)$. A map $\alpha$ from a complex manifold $X$ to $\Authol(\C^n)$ is \emph{holomorphic} 
if so is the associated map
\begin{equation*}
	X\times \C^n\ni (x,z)\mapsto \alpha(x)(z)\in \C^n,
\end{equation*}
sending $(x,z)$ to the value of $\alpha(x)\in \Authol(\C^n)$ at the point $z\in \C^n$.

\begin{dfn}\label{def:family}
Let $X$ be a complex manifold and let $G$ be a reductive group.  A \emph{holomorphic family of algebraic $G$-actions on $\C^n$ parametrized by $X$} is a holomorphic $G$-action on $X\times \C^{n}$ of the form
\begin{equation*}
\bar{\nu}\colon G\times (X\times \C^n)\ni (g,(x,z))\mapsto (x,\nu(g,x)(z))\in X\times \C^n,
\end{equation*}
for some holomorphic $\nu\colon G\times X\to\Autalg(\C^{n})$. We say that a holomorphic map $\psi\colon X\to \Authol(\C^n)$ \emph{linearizes} $\bar{\nu}$ if there is a linear representation $\rho: G \to \Gl_n (\C)$ such that, for all $g\in G$ and $x\in X$
\begin{equation}\label{eq:psi}
	\psi (x) \circ \nu(g, x) \circ \psi (x)^{-1} = \rho (g)
\end{equation}
We say that $\psi$ \emph{algebraically linearizes} $\bar{\nu}$ if $\psi$ satisfies \eqref{eq:psi} and $\psi(X)\subseteq \Autalg(\C^n)$. If there exists a holomorphic $\psi$ which (algebraically) linearizes $\bar{\nu}$, we say that $\bar{\nu}$ is (\emph{algebraically}) \emph{linearizable}.
\end{dfn}

Clearly, if a holomorphic family $\bar{\nu}$ of algebraic $G$-actions is linearizable, then every individual member $\nu (G, x)\subseteq \Autalg (\C^n) $ is linearizable, too. It is therefore natural to consider first those families for which this necessary condition is automatically satisfied. This is the case for $n=2$, when all reductive group actions are algebraically linearizable \cite[Corollary 4.4]{Kambayashi_linearization}.
\smallskip

This motivates the following theorem, which is the main result of the article.

\begin{thm}\label{thm:main} 
	Let $X$ be an open Riemann surface, let $G$ be a reductive group, and let $\bar{\nu}$ be a holomorphic family of algebraic $G$-actions on $\C^2$ parametrized by $X$.
	\begin{enumerate}
		\item\label{item:G-any} The family $\bar{\nu}$ is linearizable.
		\item\label{item:G-non-cyclic} If $G$ is not cyclic, then $\bar{\nu}$ is algebraically linearizable.
	\end{enumerate}
\end{thm}
\begin{rem}
	The question if $\bar{\nu}$ is algebraically linearizable for cyclic $G$ remains open, see Remark \ref{rem:alg_open}.
\end{rem}

\begin{rem} 
	It is also an open question if one can generalize Theorem \ref{thm:main} to $\C^{n}$ for $n\geq 3$, or to higher-dimensional parameter spaces $X$. Our proof relies both on the explicit description of $\Autalg(\C^{2})$, and on extension properties for meromorphic functions on $X$, neither of which is available in higher dimensions.
\end{rem}

Recall that the Linearization Problem for actions of finite groups on $\C^3$ is open. For a particular class of such actions, a positive answer follows from Theorem \ref{thm:main}\ref{item:G-any} applied to $X=\C^1$: 
%

\begin{cor}\label{cor:C3}
	Let $G\subseteq \Authol(\C^3)$ be a reductive group. Assume that the projection $\C^3\to\C^2$ is $G$-equivariant, and the induced $G$-action on its fibers is algebraic. Then
$G$ is holomorphically 
linearizable.
\end{cor}

\begin{rem*}
 If $G$ in Corollary \ref{cor:C3} is a subgroup of $\Autalg(\C^3)$, then by \cite{KR_families-of-group-actions}, $G$ is algebraically linearizable.
\end{rem*}

We will now outline the proof of Theorem \ref{thm:main}. 

It is known that an action of a reductive group $G$ on the affine plane $\A^{2}_{\kk}$, over an arbitrary field $\kk$, is linearizable, see \cite[4.4]{Kambayashi_linearization}, \cite[\S 7, Corollary 1]{Kraft-challenging_problems} or Lemma \ref{lem:generic_linearization}. 
Applying this result to the field $\kk=\cM(X)$ of meromorphic functions on $X$, one gets a \emph{meromorphic} map $\psi\colon X\to \Autalg(\C^{2})$ satisfying \eqref{eq:psi} for all $g\in G$ and $x\in X\setminus \{\mbox{poles of }\psi\}$; see Section \ref{sec:prelim_maps} for precise definitions.

To prove Theorem \ref{thm:main}, we need to remove the poles of $\psi$.
This is done in Section \ref{sec:proof} by replacing $\psi$ with $\psi\circ\alpha$ for a suitable meromorphic map $\alpha\colon X\to \Autalg(\C^{2})$ such that for all $x\in X$ and $g\in G$, $\alpha(x)\circ \nu(g,x)=\nu(g,x)\circ \alpha(x)$. A germ of such $\alpha$ at a pole of $\psi$ is given by Lemma \ref{lem:KR}, whose proof is based on \cite[Lemma 3.3]{KR_families-of-group-actions}. To extend such germs to a global $\alpha$, we use the following proposition. 

\begin{prop}[{see Proposition \ref{prop:Oka_explicit}}]\label{prop:Oka} 
	Let $G\subseteq \Autalg(\C^{2})$ be a reductive subgroup. Put $\Autalg^{G}(\C^2)=\{\phi\in \Autalg(\C^{2}): \forall_{g\in G}\colon\, g\circ \phi=\phi\circ g\}$. Then the basic Oka property with approximation and jet interpolation holds for those maps from connected open Riemann surfaces  to $\Autalg^G(\C^{2})$ which are of bounded degree.
\end{prop} 
For a definition of Oka properties and related notions see Section \ref{sec:Oka} or \cite[5.15]{Forstneric_book}. Proposition \ref{prop:Oka} is an equivariant version of the case $\mathscr{G}=\Autalg(\C^2)$ of \cite[Theroem 1.2]{FL_Oka-for-groups}. Its proof, given in Section \ref{sec:Oka_proof}, relies on an explicit description of the space of germs of holomorphic maps to $\Autalg^{G}(\C^2)$, given in Lemma \ref{lem:AG-is-Oka}. The latter generalizes \cite{FM_Dynamical-properties}, and makes use of the structure theorem \cite{van_der_Kulk} for $\Autalg(\C^2)$.

If $G$ is not cyclic, Lemma \ref{lem:AG-is-Oka} implies that $\Autalg^{G}(\C^{2})$ is a connected Lie group, see Lemma \ref{lem:Aut_Lie-group}. Using the usual Oka property for $\Autalg^{G}(\C^{2})$, it is easy to construct $\alpha$ which removes all poles of $\psi$ at once. This way, we prove Theorem \ref{thm:main}\ref{item:G-non-cyclic}. 

If $G$ is cyclic then $\Autalg^{G}(\C^{2})$ is more complicated. To prove Theorem \ref{thm:main}\ref{item:G-any}, we use Proposition \ref{prop:Oka} to remove poles one by one, and get $\psi$ by by passing to the limit. Therefore, our $\psi$ takes values in $\Authol(\C^{2})$ instead of $\Autalg(\C^{2})$.

\section{Preliminaries}\label{sec:prelim}

\subsection{The space of holomorphic maps to \texorpdfstring{$\Authol(\C^n)$}{Authol} and \texorpdfstring{$\Autalg(\C^n)$}{Autalg}}\label{sec:prelim_maps}
We will now recall some basic properties of the spaces of holomorphic maps from a Stein space $X$ to $\Autalg(\C^n)$ or $\Authol(\C^{n})$, and fix notation for a remaining part of the article.

As usual, we denote by $\cO(X)$ the ring of holomorphic functions on $X$, and by $\cO_{X,x}$ the ring of holomorphic germs at a point $x\in X$. We write $\cM(X)$ and $\cM_{X,x}$ for the fields of meromorphic functions on $X$ and germs at $x\in X$: they are fields of fractions of $\cO(X)$ and $\cO_{X,x}$, respectively.
\smallskip

For a ring $R$, we denote by $R^{[n]}$ the polynomial algebra in $n$ variables, and by  $\A^{n}_{R}=\Spec R^{[n]}$ the affine $n$-space over $R$, with coordinates $z\de (z_1,\dots,z_n)$. We denote the endo- and automorphism groups of $\A^{n}_{R}$ by $\End(\A^n_R)$ and $\Aut(\A^n_R)$, respectively. Thus in our notation, $\Autalg(\C^n)=\Aut(\A^{n}_{\C})$.

For $m=(m_1,\dots,m_n)\in \Z^n$  we write $z^{m}= z_1^{m_1}\cdot\dots \cdot z_n^{m_n}$ and $|m|=m_1+\dots+m_{n}$. Every $\alpha\in \Aut(\A^{n}_{R})$ can be written as $\alpha(z)=\sum_{m\in \Nn^{n}}r_{m}z^{m}$ for some $r_m\in R^{n}$. Its \emph{degree} is $\deg\alpha\de \max\{|m|:r_{m}\neq 0\}$. We will often skip the indexing set $\Nn^n$ and write simply $\alpha(z)=\sum_{m}r_{m}z^{m}$.
\smallskip

As in the introduction, we say that a map $\psi\colon X\to \Authol(\C^{n})$ is \emph{holomorphic} if the associated map 
\begin{equation*}
	\hat{\psi}\colon X\times \C^{n}\ni (x,z)\mapsto \psi(x)(z)\in \C^{n}
\end{equation*}
is holomorphic. A \emph{meromorphic} map $\psi\colon X\to \Authol(\C^n)$ is a holomorphic map $\psi\colon U\to \Authol(\C^n)$ from a Zariski open subset $U$ of $X$, such that $\hat{\psi}\colon U\times \C^n\to \C^n$ as above extends to a meromorphic map $X\times \C^n\to \C^n$. A \emph{pole} of $\psi$ is the image of a pole of $\hat{\psi}$ under the projection to $X$. Clearly, the set of poles of $\psi$ is a divisor in $X$, disjoint from $U$, and $\psi$ is holomorphic away from its poles.

Holomorphic (or meromorphic) maps to $\Autalg(\C^n)$ are the same as polynomial automorphisms of $\C^n$ with holomorphic (respectively, meromorphic) coefficients. More precisely, we have the following lemma.

\begin{lem}\label{lem:coeff_holo} Let $\psi\colon X\to \Authol(\C^n)$ be meromorphic. Write $\hat{\psi}(x,z)=\sum_{\iota}a_{\iota}(x)z^{\iota}$ for $(x,z)\in X\times \C^{n}$. Then for every $\iota\in \Nn^n$, $a_{\iota}\colon X\to \C^n$ is meromorphic, holomorphic away from the poles of $\psi$.
\end{lem}
\begin{proof}
	This result follows from the Cauchy integral formula $a_{\iota}(x)=\tfrac{1}{(2\pi i)^{n}}\int_{(\d \D)^{n}} \tfrac{\hat{\psi}(x,\zeta)}{\zeta^{\iota+1}}d\zeta$.
\end{proof}

For two meromorphic maps $\phi,\psi\colon X\to \Authol(\C^{n})$ we write $\phi\circ \psi$ for a meromorphic map $X\ni x\mapsto \phi(x)\circ \psi(x)\in\Authol(\C^{n})$. Moreover, we denote by $\id$ the constant map $X\ni x\mapsto \id_{\C^{n}}\in\Authol(\C^{n})$.
\smallskip

The space $\Authol(\C^n)$, equipped with a usual compact-open topology, admits a complete metric:
\begin{equation}\label{eq:metric}
d(\alpha,\beta)=\sum_{j=1}^{\infty}2^{-j}(\min\{\sup_{|{z}|\leq j}||\alpha({z})-\beta({z})||,1\}+\min\{\sup_{|{z}|\leq j}||\alpha^{-1}({z})-\beta^{-1}({z})||,1\});\ \alpha,\beta\in\Authol(\C^n)
\end{equation}
cf.\ \cite[4.1]{Kal-Ku_AL-theory} or \cite[p.\ 108]{Forstneric_book}. Clearly, a holomorphic map $X\to\Authol(\C^n)$ is continuous in this topology. Hence the set of all such maps is equipped with a compact-open topology, too. The following lemma summarizes some basic, well-known facts about this set.

\begin{lem}\label{lem:topology}
	Let $\cH$ denote the space of holomorphic maps from a Stein space $X$ to $\Authol(\C^n)$.
	\begin{enumerate}
		\item\label{item:composition} The map $\cH\times \cH\ni(\phi,\psi)\mapsto \phi\circ \psi \in \cH$ is continuous.
		\item\label{item:coefficients} For $\psi\in \cH$, write $\hat{\psi}(x,z)=\sum_{\iota}a_{\iota}(x)z^{\iota}$. Then $\cH\ni \psi\mapsto (a_{\iota})\in \prod_{\iota}\cO(X)^{n}$ is a homeomorphism onto its image. In other words, holomorphic maps $X\to\Authol(\C^n)$ depend continuously on their coefficients.
		\item\label{item:compacts} Let $K_{0}\subseteq K_{1}\subseteq\dots$ be an exhaustive family of compact subsets of $X$, i.e.\ $X=\bigcup_{j}K_{j}$. Then $\cH$ is completely metrizable by
		\begin{equation*}
		\mathrm{d}(\alpha,\beta)=\sum_{j=0}^{\infty} 2^{-j}\min\{d_{K_j}(\alpha,\beta),1\},
		\end{equation*}
		where for a compact set $K$ we write $d_{K}(\alpha,\beta)=\sup_{x\in K}d(\alpha(x),\beta(x))$, for $d$ as in \eqref{eq:metric}.
	\end{enumerate}
\end{lem} 
\begin{proof}
	Part \ref{item:composition} follows from the fact that $X$ and $\C^n$ are locally compact, see \cite[3.4.2]{Engelking}. Part \ref{item:coefficients} is once again a consequence of Cauchy integral formula, see Lemma \ref{lem:coeff_holo}. The formula in \ref{item:compacts} says that the compact-open topology is the topology of uniform convergence on compacts \cite[8.2.6]{Engelking}
\end{proof}


\begin{lem}\label{lem:coefficients}Let $X$ be an irreducible Stein space.
	\begin{enumerate}
		\item\label{item:bounded_degree} Let $\psi\colon X\to \Autalg(\C^n)$ be holomorphic. Then $\psi$ is \emph{of bounded degree}, that is, there is $N>0$ such that $\deg \psi(x)\leq N$ for all $x\in X$.
		\item\label{item:iso_of_schemes} The space of holomorphic (respectively, meromorphic) maps $X\to \Autalg(\C^{n})$ is naturally isomorphic to $\Aut(\A^{n}_{\cO(X)})$ (respectively, $\Aut(\A^{n}_{\cM(X)})$) as ind-groups.
	\end{enumerate}
\end{lem}
\begin{proof}
	\ref{item:bounded_degree} Write $\hat\psi(x,z)=\sum_{\iota}a_{\iota}(x)z^{\iota}$ and $X_{N}\de\{x\in X:\deg \psi(x)\leq N\}=\bigcap_{|\iota|>N}a_{\iota}^{-1}(0)$. By Lemma \ref{lem:coeff_holo}, $a_{\iota}$ is holomorphic, so each $X_{N}\subseteq X$ is closed. Because $\psi(x)\in \Autalg(\C^n)$ for every $x\in X$, we have $X=\bigcup_{N\geq 1}X_{N}$. Now, Baire category theorem implies that for some $M\geq 1$, $X_{M}$ has nonempty interior, that is, $\Int a_{\iota}^{-1}(0)\neq \emptyset$ for all $|\iota|>M$. Since $a_{\iota}$ is a holomorphic map from an irreducible space $X$, the condition  $\Int a_{\iota}^{-1}(0) \neq \emptyset$ implies that $a_{\iota}\equiv 0$. Therefore, $a_{\iota}\equiv 0$ for all $|\iota|>M$, which means that $\deg \psi(x)\leq M$ for all $x\in X$. Thus $X=X_{M}$, as claimed. 
	
	\ref{item:iso_of_schemes} Part \ref{item:bounded_degree} and Lemma \ref{lem:coeff_holo} imply that every such map can be written as a polynomial automorphism with holomorphic (resp.\ meromorphic) coefficients. This gives the required isomorphism.
\end{proof}

\subsection{Modifying germs of maps to \texorpdfstring{$\Autalg(\C^n)$}{Autalg}}\label{sec:germs}

In Section \ref{sec:KR} we will recall the method of Kraft-Russell to remove a single pole of $\psi$ in \eqref{eq:psi}. For this, we will need to work locally, with the space of germs of meromorphic maps $X\to \Autalg(\C^{n})$. By Lemma \ref{lem:coefficients}\ref{item:iso_of_schemes}, it can be identified with $\Aut(\A^{n}_{\cM_{X,x}})$. This leads to the following setting.

Let $R$ be an equicharacteristic zero discrete valuation ring with field of fractions $\kk$, maximal ideal generated by $t\in R$, and residue class field $\kappa=R/(t)$. For $\alpha\in R$ we denote by $\bar{\alpha}$ its residue mod $t$.

Denote by $v$ the $t$-adic valuation on $\kk$, and extend it to $\End(\A^n_{\kk})$ by $v(\alpha)=\min \{r\in \Z:t^{-r}\alpha\in \End(\A^{n}_{R})\}$. The following observation will allow us to replace $\alpha\in \End(\A^{2}_{\kk})$ from Lemma \ref{lem:KR} by any $\beta$ which agrees with $\alpha$ up to a sufficiently high order.

\begin{lem}\label{lem:perturbing_composition}
	Fix $\alpha_{1},\dots,\alpha_{m}\in \End(\A^{n}_{\kk})$ such that $\alpha_{m}\circ \dots\circ \alpha_{1}\in \End(\A^{n}_{R})$. Then there is $r\in \Nn$ such that for every $\beta_{1},\dots, \beta_{m}\in \End(\A^{n}_{\kk})$  satisfying $v(\alpha_{j}-\beta_{j})\geq r$, we have $\beta_{m}\circ\dots\circ\beta_{1}\in \End(\A^{n}_{R})$.
\end{lem}
\begin{proof}
	Put $v_{0}=\max\{0,-\min_{j}v(\alpha_{j})\}$. In other words, writing all coefficients of $\alpha_1,\dots,\alpha_m$ as Laurent series in $t$, $v_{0}$ is the maximal order of their poles. Write $\alpha_{j}(z)=\sum_{\iota\in I}a_{j,\iota}z^{\iota}$ for some $a_{j,\iota}\in \kk^{n}$ and a finite set $I\subseteq \Nn^{n}$ (independent of $j$). Put $U=\{1,\dots,m\}\times I\times \{1,\dots, n\}$ and $s=\#U$. By induction on $m$, one easily shows that there is a finite set $V\subseteq \Nn^{n}$ and $p_{\nu}\in (\Z^{[s]})^{n}$, $\nu\in V$, with the following property: for any $g_{j,\iota}\in \kk^{n}$, the composition $\gamma_{m}\circ\dots\circ \gamma_{1}$ of $\gamma_{j}(z)\de \sum_{\iota\in I}g_{j,\iota}z^{\iota}$ equals $z\mapsto \sum_{\nu \in V}p_{\nu}(g)z^{\nu}$. 
	
	Put $d=\max_{\nu\in V, i\in \{1,\dots,n\}}\deg (p_{\nu})_{i}$ and $r=sdv_{0}\geq 0$. Fix  $g_{j,\iota}\in \kk^{n}$ with $v(g_{j,\iota})\geq r$ for all $j\in \{1,\dots, m\}$, $\iota\in I$. We claim that $p_{\nu}(a+g)\in R^{n}$ for all $\nu\in V$. 
	Write the $i$-th coordinate of $p_{\nu}\in \Z[y_1,\dots,y_s]^{n}$ as 
	$p_{\nu}(y)_{i}=\sum_{\iota\in \Nn^{s}}c_{\iota}y^{\iota}$, $c_{\iota}\in \Z$, where $i\in \{1,\dots, n\}$. Then 
	\begin{equation*}  p_{\nu}(a+g)_{i}=\sum_{\iota\in \Nn^{s}}c_{\iota}\prod_{u\in U}(a_{u}+g_{u})^{\iota_u}=\sum_{\iota\in \Nn^{s}}c_{\iota}\prod_{u\in U}(\sum_{k=0}^{\iota_{u}}\tbinom{\iota_{u}}{k}a_{u}^{\iota_{u}-k}g_{u}^{k})=p_{\nu}(a)_{i}+\sum_{j\in J}\mu_{j} \eta_{j},\quad\eta_{j}\de\prod_{u\in U}a_{u}^{k_{u,j}}g_{u}^{l_{u,j}} 
	\end{equation*}
	for some finite $J$, $\mu_{j}\in \Z$ and $k_{u,j},l_{u,j}\in \{0,\dots, d\}$ such that for every $j\in J$, there is $u\in U$ such that $l_{u,j}\neq 0$. Thus $v(\eta_{j})=\sum_{u\in U}(k_{u,j}v(a_{u})+l_{u,j}v(g_{u}))\geq -sdv_{0}+r=0$. It follows that $\sum_{j}\mu_{j}\eta_{j}\in R$.
	
	By assumption, $p_{\nu}(a)_{i}\in R$, since it is a coefficient of $\alpha_{m}\circ\dots\circ\alpha_{1}\in\End(\A^{n}_{R})$. Therefore, $p_{\nu}(a+g)_{i}\in R$, so $p_{\nu}(a+g)\in R^{n}$, as claimed. Thus for $\beta_{j}=\alpha_{j}+\gamma_{j}$ we have  $\beta_{m}\circ \dots\circ\beta_{1}\in\End(\A^{n}_{R})$ as required.
\end{proof}

\subsection{Linearizing group actions on the affine plane}\label{sec:prelim_lin}


Given a ring extension $R\subseteq S$, we identify $\Aut(\A^{n}_{R})$ with a subgroup of $\Aut(\A^n_{S})$ via base change $\Spec S\to \Spec R$. We treat $\Gl_{n}(R)$ as a subgroup of $\Aut(\A^{n}_{R})$, too. 

Let $\kappa$ be a field, let $G$ be a reductive, $\kappa$-linear group, and let $R$ be a ring containing $\kappa$. Assume that $G\subseteq \Aut(\A^{n}_{R})$. We say that $\psi\in \Aut(\A^{n}_{R})$ \emph{linearizes $G$ over $R$} if
\begin{equation}\label{eq:lin_R}
\mbox{there is a representation } \rho\colon G\to\Gl_{n}(\kappa) \mbox{ such that }\psi\circ g \circ \psi^{-1}=\rho(g) \mbox{  for every } g\in G.
\end{equation}
We say that $G$ is \emph{linearizable over $R$} if there is such $\psi$. Note that for $\kappa=\C$ and a Stein space $X$, $G$ is linearizable over $\cO(X)$ if and only if the family $G\times (X\times \C^n)\ni (g,(x,z))\mapsto (x,g(x)\cdot z)$  is algebraically linearizable in the sense of Definition \ref{def:family}.
\smallskip

We now restrict our attention to $n=2$. Then the Jung--van der Kulk theorem \cite{Jung,van_der_Kulk}, see \cite[Theorem 3.3]{Nagata_Aut-C^2}, asserts that for any field $\kk$, $\Aut(\A^{2}_{\kk})$ is an amalgamated product of the affine group $\Aff\de \Aff_{2}(\kk)=\kk_{+}\rtimes \Gl_{2}(\kk)$ and the group of \emph{elementary automorphisms}
\begin{equation*}
E=\{(z_1,z_2)\mapsto (\alpha z_1+p(z_2),\beta z_2 +\beta'): \alpha,\beta\in \kk^{*},\beta'\in \kk,\ p\in \kk^{[1]}\}.
\end{equation*}
For a subgroup $G\subseteq \Aut(\A^{2}_{\kk})$, we write $\Aut^{G}(\A^{2}_{\kk})=\{\phi\in \Aut(\A^{2}_{\kk}):\forall_{g\in G}\, \phi\circ g=g\circ \phi\}$.  If $H\subseteq \Aut(\A^{2}_{\kk})$, we denote by $H^{G}\de H\cap\Aut^{G}(\A^{2}_{\kk})$ the centralizer of $G$ in $H$: in particular, $\Aff^{G}= \Aff\cap \Aut^{G}(\A^{2}_{\kk})$, $E^{G}= E\cap \Aut^{G}(\A^{2}_{\kk})$, $\Gl_{2}^{G}(\kk)= \Gl_{2}(\kk)\cap \Aut^{G}(\A^{2}_{\kk})$, etc.
\smallskip

The following lemma is a well known consequence of van der Kulk theorem, see \cite[Corolary 4.4]{Kambayashi_linearization} or \cite[Proposition 2.2]{Kuroda_subgroups}. We sketch the argument in our setting following \cite[Lemma 3.1]{KR_families-of-group-actions}. The assumption that $G$ is split is automatically satisfied e.g.\ if $\kappa$ is algebraically closed \cite[17.16]{Milne_RG}, which will be the case when  $\kappa=\C$, $R=\cO(X)$ or $\cO_{X,x}$.

\begin{lem}\label{lem:generic_linearization}
	Let $\kappa\subseteq \kk$ be a field extension, and let $G$ be a split reductive $\kappa$-linear group. Assume that $G\subseteq \Aut(\A^{2}_{\kk})$. Then $G$ is linearizable over $\kk$.
\end{lem}
\begin{proof}
	Since $G\subseteq \Aut(\A^{2}_{\kk})$ is algebraic, it has bounded length as a subgroup of $\Aff * E$, which by \cite[Theorem 8]{Serre_trees} implies that $\phi^{-1}\circ G\circ \phi\subseteq \Aff$ or $E$ for some $\phi\in \Aut(\A^{2}_{\kk})$, see \cite[Corollary 1]{Kraft-challenging_problems}. Because $G$ is reductive, it follows that $\phi^{-1}\circ G\circ \phi\subseteq \Gl_{2}(\kk)$. Since by assumption $G$ is split over $\kappa$, by \cite[20.7]{Milne_RG} the representation $G\ni g\mapsto \phi^{-1}\circ g \circ \phi\in \Gl_{2}(\kk)$ is isomorphic, over $\kk$, to the extension of some representation $\rho\colon G\to \Gl_{2}(\kappa)$. More explicitly, there is $\alpha\in \Gl_{2}(\kk)$ such that $\alpha^{-1}\circ  (\phi^{-1}\circ g\circ \phi) \circ \alpha=\rho(g)\in \Gl_{2}(\kappa)$ for all $g\in G$. Thus $\psi\de \phi\circ \alpha$ linearizes $G$ over $\kk$.
\end{proof}

Once $\kk$ gets replaced by a ring $R$, no analogue of van der Kulk theorem is available, and Lemma \ref{lem:generic_linearization} becomes a difficult problem. It was solved by Sathaye \cite{Sathaye_DVR} if $R$ is a discrete valuation ring, and by Kraft and Russell \cite[Theorem D]{KR_families-of-group-actions} if $R$ is a coordinate ring of a factorial affine curve. More generally, \cite[Theorem 1.1]{Kuroda_subgroups} solves it in case $R$ is a PID or even UFD (under some assumptions).


\section{Removing poles after Kraft--Russell--Sathaye}\label{sec:KR}

In this section, we prove a result which, applied to $R=\cO_{X,x}$, will allow us to remove the pole at $x\in X$ of $\psi$ satisfying \eqref{eq:psi}. We follow the approach of Kraft and Russell \cite[Lemma 3.3]{KR_families-of-group-actions},  based on an algebraic lemma due to Sathaye \cite{Sathaye_DVR}. A similar argument is used in the proof of \cite[Theorem 1.1(i)]{Kuroda_subgroups}, see Lemma 3.1 loc.\ cit. for the key claim. 

We keep the notation introduced in Section \ref{sec:germs}.

\setcounter{claim}{0}
\begin{lem}
	\label{lem:KR}
	Let $R$ be an equicharacteristic zero discrete valuation ring with field of fractions $\kk$ and residue field $\kappa\subseteq R$. Let $G$ be a reductive $\kappa$-linear subgroup of $\Aut (\A^{2}_{R})$. Assume that there is $\psi\in \Aut(\A^{2}_{\kk})$ which linearizes $G$ over $\kk$, see \eqref{eq:lin_R}. Then there is $\alpha\in \Aut^{G}(\A^{2}_{\kk})$ such that $\psi\circ \alpha\in \Aut(\A^{2}_{R})$.
\end{lem}
\begin{proof}
    By \eqref{eq:lin_R}, there is a representation $\rho\colon G\to \Gl_{2}(\kappa)$ such that $\psi\circ g \circ \psi^{-1}=\rho(g)$ for all $g\in G$.
    
	Define $\beta\in \Aut(\A^{2}_{\kk})$ by $\beta(z_1,z_2)=(t^{-v(\psi)}z_1,t^{-v(\psi)}z_2)$. Then $\beta\circ \rho(g)=\rho(g) \circ\beta$, so $\alpha\de \psi^{-1}\circ\beta\circ\psi\in \Aut^{G}(\A^{2}_{\kk})$. Replacing $\psi$ by $\psi\circ \alpha$, we can assume $v(\psi)=0$, so $\psi\in \End(\A^{2}_{R})$. Assume $\psi\not\in\Aut(\A^2_R)$.
	

	\begin{claim*}\label{cl:N}
		There is $\tau\in \Aut(\A^{2}_{\kappa})$ such that $t|(\tau\circ \psi)^{*}z_1$ and for every $g\in G$,  $\tau\circ\rho(g)\circ \tau^{-1} \in \Gl_{2}(\kappa)$ is diagonal.
	\end{claim*}
	\begin{proof}
	Since $\psi\in \End(\A^2_{R})$, taking residues mod $t$ gives a proper map $\bar{\psi}\colon \A^{2}_{R}\to \A^{2}_{\kappa}$, which is equivariant with respect to the action of $G$ on the source and $\rho(G)$ on the target. The image $C\de \bar{\psi}(\A^{2}_{R})$ is $\rho(G)$-invariant, and $\dim C=1$: indeed, $\dim C>0$ since  $v(\psi)=0$, and $\dim C\neq 2$ since $\psi\not\in\Aut(\A^2_{R})$. Let $f\in \kappa^{[2]}$ be the generator of the ideal of $C$. Since $C$ is $\rho(G)$-invariant, the line $\kappa\cdot f\subseteq \kappa^{[2]}$ is $\rho(G)$-invariant, too, i.e.\ for every $g\in G$, $f\circ \rho(g)= \lambda(g)\cdot f$ for some homomorphism $\lambda\colon G\to \kappa^{*}$. 
	
	By \cite[Remark 2.1]{Sathaye_DVR}, $\kappa[C]=\kappa^{[2]}/(f)\cong \kappa^{[1]}$. The Abhyankar--Moh--Suzuki theorem \cite{AbhMoh_the_line_thm,Suzuki_AMSthm} implies that $f$ is a coordinate of $\kappa^{[2]}$, i.e. $\kappa^{[2]}=\kappa[f,h]$ for some $h\in \kappa^{[2]}$. Put $\tau\de (f,h)\in \Aut (\A ^2_{\kappa})$. Then $\tau^{*}z_{1}=f\in \ker \bar{\psi}^{*}$, so $t|\psi^{*}\tau^{*}z_1=(\tau\circ\psi)^{*}z_1$, as required. It remains to prove that $\tau\circ\rho(g)\circ\tau^{-1}$ is diagonal.
	\smallskip
	
	Consider the case $\deg f=1$. Then $C$ is an eigenspace of the linear action of $\rho(G)$. Since $G$ is reductive, we can take $h$ to be a generator of another eigenspace, i.e.\ $\deg h=1$ and $h\circ \rho(g)=\mu(g) \cdot h$ for another homomorphism $\mu\colon G\to \kappa^{*}$. Thus $\tau \circ \rho(g)=(\lambda(g)f,\mu(g)h)=\delta(g)\circ \tau$ for a diagonal $\delta(g)=(\lambda(g),\mu(g))$, as needed. 
	
	Consider the case $\deg f>1$. Then \cite[Theorem 8.5, Chapter 6]{Cohn} shows that replacing, if necessary, $h$ by $h-s f^{d}$ for some $s\in \kappa^{*}$ and $d\in \Nn$ , we can assume $\deg h<\deg f$. Now for any $g\in G$, we have $\tau\circ\rho(g)=(\lambda(g)f,h_{g})$, where $h_{g}\de \rho(g)^{*}h$ has the same properties as $h$, that is, $\deg h_{g}=\deg h<\deg f$ and $\kappa[f,h]=\kappa[f,h_{g}]$. It follows that $h_{g}=\mu(g)(h)$ for some homomorphism $\mu\colon G\to \Aff_{1}(\kappa)$. Because $G$ is reductive, $\mu(G)$ is conjugate to a subgroup of $\kappa^{*}$. Replacing $h$ by some $h+c$, $c\in \kappa$ we can assume that in fact $\mu(G)\subseteq \kappa^{*}$. As before, we conclude that $\tau\circ\rho(g)=\delta(g)\circ \tau$ for a diagonal $\delta(g)\de (\lambda(g),\mu(g))$.
\end{proof}

	Let us now introduce the following notation. For $\phi\in \Aut(\A^{2}_{\kk})\cap \End(\A^2_{R})$ and $j\in \{1,2\}$ put  $w_{j}(\phi)=-v((\phi^{-1})^{*}z_j)$, that is, $w_{j}(\phi)$ is the minimal number $w$ such that $t^{w}z_{j}$ is contained in the pullback $\phi^{*}R[z_1,z_2]\subseteq R[z_1,z_2]$. The latter inclusion holds since $\phi\in \End(\A^2_{R})$. Note that if $\tau\in \Aut(\A^{2}_{R})$ then $w_{j}(\tau\circ \phi)=w_{j}(\phi)$. Put $w(\phi)=w_{1}(\phi)+w_{2}(\phi)$. Then $w(\phi)\geq 0$, with equality if and only if $\phi\in \Aut(\A^{2}_{R})$.

	We need to show that there exists $\alpha\in \Aut^{G}(\A^{2}_{\kk})$ such that $\psi\circ\alpha\in \End(\A^{2}_{R})$ and $w(\psi\circ\alpha)<w(\psi)$. The lemma will then follow by induction on $w(\psi)$. 
	\smallskip

	Since $\psi\not\in \Aut(\A^2_{R})$, we have $w(\psi)>0$, so, say, $w_{1}(\psi)>0$. Let $\tau$ be as in the Claim above. Define 
	\begin{equation*}
		\gamma\in \Aut(\A^{2}_{\kk}), \quad \gamma(z_1,z_2)=(t^{-r_1}z_1,t^{-r_2}z_2),\quad\mbox{where}\quad r_j=v((\tau \circ \psi)^{*}z_{j});
	\end{equation*}
	and put $\alpha\de (\tau\circ \psi)^{-1}\circ \gamma\circ (\tau\circ \psi)\in \Aut(\A^{2}_{\kk})$. Now \eqref{eq:lin_R} and the fact that $\gamma$ commutes with $\tau\circ \rho(g)\circ\tau^{-1}$ since they are both diagonal, imply that  $\alpha\in \Aut^{G}(\A^{2}_{\kk})$. Indeed, $\tau^{-1}\circ \gamma\circ \tau\circ\rho(g)=\rho(g)\circ \tau^{-1}\circ \gamma\circ\tau$, so
	\begin{equation*}	
		\alpha\circ g =  
		\alpha \circ \psi^{-1} \circ \rho(g)\circ \psi =
		\psi^{-1}\circ\tau^{-1}\circ \gamma\circ \tau\circ\rho(g)\circ\psi=
		\psi^{-1}\circ \rho(g)\circ \tau^{-1}\circ \gamma\circ\tau\circ \psi=
		g\circ \alpha.
	\end{equation*}
	Put $w_j=w_j(\psi)$, $w_j'=w_j(\psi\circ\alpha)$. Since $\tau\in \Aut(\A^2_{\kappa})\subseteq \Aut(\A^2_R)$, we have $w_j=w_j(\tau\circ\psi)$ and $w'_{j}=w_{j}(\tau\circ\psi\circ\alpha)=w_{j}(\gamma\circ\tau\circ\psi)$. 
	Definition of $r_j$ implies that $w_{j}'\geq 0$, so $\psi\circ\alpha\in \End(\A^2_{R})$. Since $\psi\in \End(\A^{2}_{R})$, too, we have $r_j\geq 0$, hence $w_j'\leq w_j(\tau\circ\psi)=w_j$. It remains to show that, say, $w_1'< w_1$.
	
	 Write $\tau\circ\psi(z_1,z_2)=(\tilde{y}_1,\tilde{y}_2)$ and $\gamma\circ\tau\circ\psi(z_1,z_2)=(y_1,y_2)$. Since $t^{w_1}z_1\in (\tau\circ \psi)^{*}R[z_1,z_2]$, we have
	\begin{equation}\label{eq:w}
	t^{w_1}z_1=\sum_{i,j\geq 0} a_{ij}\tilde{y}_1^{i}\tilde{y}_2^{j}=
	\sum_{i,j\geq 0} a_{ij}t^{r_{1}i+r_{2}j}y_{1}^{i}y_{2}^{j}
	\end{equation}
	for some $a_{ij}\in R$. 
	We have $w_1>0$ by assumption, and $r_1>0$ by the Claim. If $r_2>0$, too, then \eqref{eq:w} gives $\bar{a}_{00}=0$, and dividing \eqref{eq:w} by $t$ we infer that $t^{w_1-1}z_1\in R[y_{1},y_{2}]=(\gamma\circ\tau\circ\psi)^{*}R[z_1,z_2]$, so $w'_1\leq w_1-1$, as claimed. Assume $r_2=0$. By definition of $r_2$, $t\nmid (\gamma\circ \tau\circ \psi)^{*}z_{2}$, i.e.\ $\bar{y}_2\neq 0$. Now \eqref{eq:w} gives $\sum_{j} \bar{a}_{0j}\bar{y}_{2}^{j}=0$, so $\bar{a}_{0j}=0$ for all $j$. Again, dividing \eqref{eq:w} by $t$ gives $w_1'\leq w-1$; which ends the proof.
\end{proof}

\begin{rem}[{cf.\ \cite[Remark 3.1]{KR_families-of-group-actions}}]\label{rem:deg_alpha}
	The most difficult part of Lemma \ref{lem:KR}, when \cite{AbhMoh_the_line_thm} and \cite{Sathaye_DVR} are used, is the case when $\rho(G)$ acts on the curve $C=\bar{\psi}(\A^{2}_{\kk})\cong \A^{1}_{\kappa}$, of degree at least two in $\A^{2}_{\kappa}$. The action of $\rho(G)$ on $C$ is faithful, since it is such on $\sspan_{\kappa}(C)=\A^{2}_{\kappa}$. Thus in this case, $G$ embeds in $\kappa^{*}$.
	
	On the other hand, if $G\not\subseteq \kappa^{*}$, then the required $\alpha$ is constructed before the Claim. The corresponding $\rho(G)$-equivariant automorphism $\beta\de \psi\circ \alpha\circ \psi^{-1}$ is \emph{affine} (in fact, diagonal). In Lemma \ref{lem:Aut_Lie-group} we will see that in this case \emph{all} $\rho(G)$-equivariant automorphisms of $\A^{2}_{\kk}$ are affine.
	
	However, if $G\subseteq \kappa^{*}$ then the corresponding $\beta\in \Aut^{\rho(G)}(\A^{2}_{\kappa})$ is constructed inductively, by composing maps of the form $\tau^{-1}\circ \gamma\circ \tau$, where $\gamma$ is affine, but $\tau$ comes from the Abhyankar--Moh--Suzuki theorem. Thus $\deg\beta$ depends both on the degrees of $\tau$'s, and  on the number of inductive steps, which in turn depends on the order $w(\psi)$ of the pole of $\psi^{-1}$. Therefore, it seems that this method cannot be used to remove infinitely many poles at once, see Remark \ref{rem:alg_open}.
\end{rem}	

\section{Equivariant automorphisms of the affine plane}\label{sec:FM}

In this section, we use the Jung -- van der Kulk theorem to investigate the structure of $\Aut^{G}(\A^{2}_{\kk})$ in more detail. Our approach extends the one of Friedland and Milnor \cite{FM_Dynamical-properties} to the equivariant setting. 

Here, $\kk$ is an arbitrary field of characteristic zero. The results of \cite{FM_Dynamical-properties} are stated over $\kk=\R$ or $\C$, but those which are relevant for us hold over $\kk$, too. We keep the notation introduced in Section \ref{sec:prelim_lin}.
\smallskip

By \cite{van_der_Kulk}, any $\phi\in \Aut(\A^{2})\setminus \Aff$ can be written as $a_{m+1}\circ e_{m}\circ a_{m}\circ\dots\circ e_{1}\circ a_{1}$ for some $m\geq 0$, with elementary $e_{j}\in E\setminus \Aff$ and affine $a_{j}\in \Aff$, such that $a_{2},\dots, a_{m}\not\in E$. Moreover, this decomposition is unique up to replacing pairs $(e_{j},a_{j})$ by $(e_{j}\circ s, s^{-1}\circ a_{j})$ for some $s\in \Aff\cap E$. Therefore, the sequence  $(\deg e_{1},\dots, \deg e_{m})$, called a \emph{polydegree} of $\phi$, does not depend on the decomposition, and is invariant under a linear change of coordinates. For convenience, we define a polydegree of elements of $\Aff$  as an empty sequence. The set of all automorphisms of polydegree $d=(d_{1},\dots, d_{m})$ will be denoted by $\cA_{d}$. 

For every $\phi\in \Aut(\A^{2})$ of degree at least $2$, there is a unique line $\ll_{\phi}$ through the origin such that  $\deg\phi(\ll_{\phi})<\deg \phi$. By \cite[p.\ 72]{FM_Dynamical-properties}, for each nonempty polydegree $d$, the map
\begin{equation}\label{eq:F-bundle}
\gamma_{d} \colon \cA_{d} \ni \phi \mapsto (\ll_{\phi^{-1}},\ll_{\phi})\in \P^{1}\times \P^{1}
\end{equation}
is a locally trivial fiber bundle. To write down its sections, parametrize $\P^{1}$ so that the lines $\{z_1=0\}$, $\{z_2=0\}$ correspond to $\infty,0\in \P^{1}$. More precisely, fix isomorphisms $\A^{1}\ni \lambda\mapsto \{z_2=\lambda z_1\}\in U_0=\P^{1}\setminus \{\infty\}$ and $\A^{1}\ni \lambda\mapsto \{z_2=\lambda z_1\}\in U_1=\P^{1}\setminus \{0\}$. Put $\cF_{d}=\gamma_{d}^{-1}\{(0,0)\}$. Then sections of $\gamma_d$ can be written as 
	\begin{equation}\label{eq:section}
	U_{\eta}\times U_{\theta}\times \cF_{d} \ni (\lambda,\mu,f)\mapsto a_{\eta, \lambda}\circ f \circ a_{\theta,\mu}^{-1} \in \cA_d,
	\end{equation}
where for $\lambda \in \kk$ we put $a_{0,\lambda}(z_1,z_2)=(z_1,\lambda z_1+z_2)$ and $a_{1,\lambda}(z_1,z_2)=(\lambda z_1+z_2,z_2)$. 

To describe the fiber $\cF_{d}$ over $(0,0)$, we introduce the following notation. First, put 
\begin{equation*}
\hat{S}= \{(z_1,z_2)\mapsto (\alpha z_1 +\alpha', \beta z_2 +\beta') : \alpha,\beta\in \kk^{*}, \alpha',\beta'\in \kk\}\subseteq \Aff\cap E.
\end{equation*} 
Next, for $q\in \kk^{[1]}$ define $e_{q}\in E$ by
\begin{equation*}
	e_{q}(z_1,z_2)=(z_1+z_2q(z_2),z_2).
\end{equation*}
For $s\in \hat{S}$ and an $m$-tuple $q=(q_{1},\dots, q_{m})\in (\kk^{[1]})^{m}$, put
\begin{equation*}
f_{s,q}=s\circ e_{q_m}\circ \tau\circ e_{q_{m-1}}\circ \tau \dots \tau\circ e_{q_1}, \quad \mbox{where } \tau\colon (z_1,z_2)\mapsto (z_1,z_2).
\end{equation*}
Now,  \cite[Lemma 2.10]{FM_Dynamical-properties} asserts that for a polydegree $d=(d_{1},\dots, d_{m})$
\begin{equation}\label{eq:F}
\cF_{d}=\{f_{s,q}: s\in \hat{S},\ q=(q_1,\dots, q_m)\in (\kk^{[1]})^{m},\ \deg q_{j}=d_{j}-1\},
\end{equation}
and $s,q$ are determined uniquely by $f_{s,q}\in \cF_{d}$. 
\smallskip

Having introduced the above notation, we pass to the equivariant setting. For $G\subseteq \Gl_{2}(\kk)$, write
$\cA_{d}^{G}=\cA_{d}\cap\Aut^{G}(\A^{2})$, $\cF_{d}^{G}=\cF_{d}\cap \Aut^{G}(\A^{2})$.
By \eqref{eq:F}, $\cF_{d}\cong (\A^{1}_{*})^{N}\times \A^{M}$ for some $N,M\geq 0$. Our goal is to show that $\cF^{G}_{d}$ is a vanishing set of some of these coordinates.

We will use the following subgroups of $\hat{S}$:
\begin{equation*}\begin{split}
T&= \{(z_1,z_2)\mapsto (\alpha z_1 +\alpha', \beta z_2) : \alpha,\beta\in \kk^{*}, \alpha'\in \kk\}\\
D&= \{(z_1,z_2)\mapsto (\alpha z_1, \beta z_2) : \alpha,\beta\in \kk^{*}\},\quad Z=\{(z_1,z_2)\mapsto (\alpha z_1, \alpha z_2) : \alpha\in \kk^{*}\}=Z(\Gl_{2}(\kk)).
\end{split}
\end{equation*}

\begin{lem}\label{lem:gf}
	Fix $g\colon (z_1,z_2)\mapsto (\lambda_0 z_1,\lambda_1 z_2)$, $s\colon (z_1,z_2)\mapsto (\alpha z_1+\alpha',\beta z_2 + \beta')\in \hat{S}$ and $q\in (\kk^{[1]})^{m}$. Then 
	\begin{equation*}
	g^{-1}\circ f_{s,q} \circ g=f_{\sigma,p}\quad \mbox{for}\quad 
	\sigma(z_1,z_2)=(\tfrac{\alpha \lambda_{m-1}}{\lambda_{0}} z_1+\tfrac{\alpha'}{\lambda_{0}}, \tfrac{\beta\lambda_{m}}{\lambda_{1}}z_2+\tfrac{\beta'}{\lambda_{1}})\quad\mbox{and}\quad
	p_{j}(z_2)=\tfrac{\lambda_{j}}{\lambda_{j+1}}q_{j}(\lambda_{j}z_2)
	\end{equation*}
	for $j\in \{1,\dots, m\}$, where we put $\lambda_{j}=\lambda_0$ if $2|j$ and $\lambda_j=\lambda_{1}$ if $2\nmid j$.
\end{lem}
\begin{proof}
Put $g_{j}=(\lambda_{j}z_1,\lambda_{j+1}z_2)$, so $g_{0}=g$. A direct computation shows that for $j\in \{1,\dots, m\}$ we have $e_{q_{j}}\circ g_{j-1}=g_{j-1}\circ e_{p_{j}}$, $\tau\circ g_{j-1}=g_{j}\circ \tau$ and $s\circ g_{m-1}=g_{0}\circ \sigma$. The result follows.	
\end{proof}

If $\kk$ contains a $k$-th root of unity, we fix a primitive one $\zeta_{k}$.

\begin{lem}\label{lem:G}
Fix $\{\id\} \neq G\subseteq D$ and a polydegree $d=(d_1,\dots,d_m)$ with $m>0$. Assume $\cF_{d}^{G}\neq \emptyset$. Then 
\begin{equation}\label{eq:G}
G=H_{d_{1}}\de \{(z_1,z_2)\mapsto(\lambda^{d_{1}}z_1,\lambda z_2) :\lambda\in H\} \quad \mbox{for some nontrivial subgroup } H\subseteq \kk^{*}.
\end{equation}
If $H$ is not finite then $m=1$ and $\cF_{d}^{G}=\{s\circ e_{\lambda z_2^{d_1-1}}: s\in D,\ \lambda\in\kk^{*}\}$.
  
Assume $H=\langle \zeta_{k} \rangle$. If $k|d_{1}$ then $m=1$ and 
\begin{equation}\label{eq:with_T}
\cF^{G}_{d}=\{s\circ e_{q_1}\in \cF_{d}: s\in T,\ \exists_{\hat{q}_1\in \kk^{[1]}}\colon\, q_{1}(z_2)=z_2^{k-1}\hat{q}_{1}(z_2^{k}) \}.
\end{equation}

Assume that $k\nmid d_{1}$. Then $k|d_{1}-1$ (i.e.\ $G\subseteq Z$), or $2\nmid m$. Moreover, if $m\geq 2$ then $\gcd(k,d_1)=1$. For $j\in \{1,\dots, m\}$ define $l_{j}$ as $d_{1}$ mod $k$ if $2\nmid j$ and as the inverse of $d_{1}$ mod $k$ if $2|j$. Then
\begin{equation*}
\cF^{G}_{d}=\{f_{s,q}\in \cF_{d}:s\in D,\ \forall_{j\in \{1,\dots, m\}}\exists_{\hat{q}_j\in \kk^{[1]}}\colon\, q_{j}(z_2)=z_2^{l_{j}-1}\hat{q}_{j}(z_2^{k}) \}.
\end{equation*}
\end{lem}
\begin{proof}
	 Fix $g\in G$, $g(z_1,z_2)=g(\lambda_0 z_1,\lambda_1 z_2)$ for some $\lambda_0,\lambda_1\in \kk^{*}$ as in Lemma \ref{lem:gf}. By \eqref{eq:F}, every element $f\in \cF_{d}$ can be written uniquely as $f_{s,q}$, so to check when $f\in \cF^{G}_{d}$ we can compare $s,q$ with $\sigma,p$ from Lemma \ref{lem:gf}. Comparing the leading coefficients of  $q_{1}$ and $p_{1}$ yields $\lambda_{0}=\lambda_{1}^{d_{1}}$. This proves the first assertion. We need to show that if $m\geq 2$ or if $q_1$ is not a scalar multiple of a monomial, then $\lambda_{1}$ is a root of unity.
	
	If $q_{1}$ is not a multiple of a  monomial, so $q_{1}(z_2)$ has a nonzero coefficient near some $z_2^{r-1}$ for $r\neq d_{1}$, then $\lambda_{1}^{d_{1}}=\lambda_{0}=\lambda_{1}^{r}$, so $\lambda_{1}^{d_{1}-r}=1$. If $m\geq 2$ then comparing $q_2$ with $p_2$ we see that $\lambda_{1}=\lambda_{0}^{d_{2}}=\lambda_{1}^{d_{1}d_{2}}$, hence $\lambda_{1}$ is a $k$-th root of unity for some $k|d_{1}d_{2}-1$. Note that in this case $\gcd(k,d_1)=1$. 
	
	Thus $G=\langle g\rangle$ with $g(z_1,z_2)=(\zeta^{d_{1}}_{k}z_1,\zeta_{k}z_2)$, where $\gcd(k,d_1)=1$ if $m\geq 2$. 
	The set $\cF^{G}_{d}$ consists of those $f_{s,q}\in \cF_{d}$ for which $q=p$ and $s=\sigma$, where $p,\sigma$ are as in Lemma \ref{lem:gf}: indeed, since $G$ is cyclic, it suffices to check equivariance only for the generator $g$. If the coefficient of $q_{j}(z_2)$ near $z_2^{r}$ is nonzero then comparing $q_j$ with $p_j$ gives $\lambda_{j+1}=\lambda_{j}^{r+1}$. If $2\nmid j$, this condition means that $\zeta_{k}^{d_1}=\zeta_{k}^{r+1}$, so $r\equiv d_{1}-1 \equiv l_{j}-1 \pmod{k}$. If $2|j$, we get $\zeta_{k}=(\zeta_{k}^{d_{1}})^{r+1}$, so $(r+1)d_{j}\equiv 1\pmod{k}$, and again $r\equiv l_{j}-1\pmod{k}$.

	Comparing $s$ with $\sigma$ we get $2\nmid m$ or $\lambda_0=\lambda_1$, i.e.\ $G\subseteq Z$. Moreover, since $\lambda_{1}=\zeta_{k}\neq 1$, we see that $s=\sigma$ if and only if $s\in D$ for $k\nmid d_{1}$ and $s\in T$ for $k|d_{1}$. Note that the latter is possible only if $m=1$, for otherwise $\gcd(k,d_1)=1$.
\end{proof}
\begin{rem}\label{rem:FG-is-Oka}
	By \eqref{eq:F}, the map associating to each $f_{s,q}\in\cF_{d}$ the coefficients of $s$ and $q$ is an isomorphism $\cF_{d}\cong (\A^{1}_{*})^{M}\times \A^{N}$ for some $M,N\geq 0$. Lemma \ref{lem:G} shows that a subvariety $\cF_{d}^{G}\subseteq \cF_{d}$ is  defined by vanishing of a certain subset of these coordinates. In particular, $\cF_{d}^{G}\cong (\A_{1}^{*})^{M'}\times \A^{N'}$ for some $M',N'\geq 0$.
\end{rem}

\begin{lem}\label{lem:AG-is-Oka}
	For $\{\id\}\neq G\subseteq \Gl_{2}(\kk)$ and a nonempty polydegree $d=(d_{1},\dots, d_{m})$ let $\gamma^{G}_{d}\colon \cA^{G}_{d}\to \P^{1}\times\P^{1}$ be the restriction of $\gamma$ from \eqref{eq:F-bundle}. Assume that $\cA_{d}^{G}\neq \emptyset$. Then after some linear change of coordinates, $G=H_{d_{1}}$, see  \eqref{eq:G}, for a subgroup $H\subseteq\kk^{*}$, and exactly one of the following holds.
	\begin{enumerate}
		\item\label{item:scalar} 
		$H=\langle \zeta_{k} \rangle$ for some $k|d_{1}-1$. Then $\gamma^{G}_{d}$ endows $\cA_{d}^{G}$ with a structure of an $\cF_{d}^{G}$-bundle over $\P^{1}\times \P^{1}$.
		\item\label{item:cyclic} 
		$H=\langle \zeta_{k} \rangle$ 
		for some $k$ coprime to $d_{1}$ and $k\nmid d_{1}-1$. Then $\cA_{d}^{G}$ is a disjoint union of two isomorphic  fibers of $\gamma^{G}$, namely $\cF_{d}^{G}$ over $(0,0)$ and $\{\tau\circ f \circ \tau:f\in \cF_{d}^{G}\}$ over $(\infty,\infty)$. 
		\item\label{item:other}  $H$ is none of the above groups. 
		Then $\cA_{d}^{G}=\cF_{d}^{G}$.
	\end{enumerate}
\end{lem}
\begin{proof}
	Assume first that $G\subseteq Z$. Then by Lemma \ref{lem:G} $G$ is as in \ref{item:scalar}.  Moreover, since all $g\in G$ commute with $a\in \Gl_{2}(\kk)$, by \eqref{eq:section} $\gamma$ is $G$-equivariant, and $\gamma^{G}$ is an $\cF_{d}^{G}$-bundle over $\P^{1}\times \P^{1}$, as claimed.
	
	Assume now that $G\not\subseteq Z$. Recall that for $\phi\in \cA_{d}$, $\ll_{\phi}$ denotes the unique line such that $\deg \phi(\ll_{\phi})<\deg \phi$. Now if $\phi\in \cA_{d}^{G}$ then for every $g\in G$ we have $g\circ \phi=\phi\circ g$, so $\ll_{\phi}=\ll_{g\circ \phi}=\ll_{\phi\circ g}=g^{-1}(\ll_{\phi})$, that is, $\ll_{\phi}$ is an eigenspace of every $g\in G$. After a linear change of coordinates, we can assume $\ll_{\phi}=\{z_2=0\}$.
	
	Consider the case when $\cA_{d}^{G}$ contains some $\psi$ with $\ll_{\psi}\neq \ll_{\phi}$. Because all elements of $G$ share the same eigenspaces $\ll_{\phi}$, $\ll_{\psi}$, fixing coordinates such that $\ll_{\psi}=\{z_1=0\}$ we see that $G\subseteq D$. On the other hand, since $G\not\subseteq Z$, at least one $g\in G$ has $\ll_{\phi}$, $\ll_{\psi}$ as its only invariant lines. In the parametrization of $\P^{1}$ as in \eqref{eq:section}, $\ll_{\phi}$ and $\ll_{\psi}$ correspond to $0$ and $\infty$, respectively. Therefore, for every  $d$, $\cA_{d}^{G}$ is contained in the union of $\cF_{d}=\gamma^{-1}\{(0,0)\}$, $\cF_{d}^{\infty}=\gamma^{-1}\{(\infty,\infty)\}$, $\cF_{d}^{+}=\gamma^{-1}\{(0,\infty)\}$ and  $\cF_{d}^{-}=\gamma^{-1}\{(\infty,0)\}$. 
	
	Since $\cA_{d}^{G}$ is nonempty and $G\not\subseteq Z$, by Lemma \ref{lem:G} $2\nmid m$. Since $\cF_{d}^{-}=\{f\circ \tau:f\in \cF_{d}\}=\{f_{s,q'}:f_{s,q}\in \cF_{d}\}$, where $q'=(0,q)$, comparing $s$ with $\sigma$ from Lemma \ref{lem:gf} as in the proof of Lemma \ref{lem:G}, we infer from $2\nmid m$ that $\cF_{d}^{-}\cap \cA_{d}^{G}=\emptyset$. Similarly, since $\cF_{d}^{+}=\{f_{s'',q''}:f_{s,q}\in \cF_{d}\}$, where $s''=\tau\circ s \circ \tau$ and $q''=(q,0)$, we get $\cF_{d}^{+}\cap \cA_{d}^{G}=\emptyset$, too. Because by assumption $\ll_{\psi}\neq \ll_{\phi}$ for some $\psi\in \cA_{d}^{G}$, we have $\emptyset \neq \cF^{\infty}_{d}\cap \cA_{d}^{G}=\{\tau\circ f \circ \tau: f\in \cF_{d}^{G^{\tau}}\}$, where $G^{\tau}=\{\tau\circ g \circ \tau:g\in G\}\subseteq D$. Hence by Lemma \ref{lem:G}, both $G$ and  $G^{\tau}$ are as in \eqref{eq:G}. The fact that $G\not\subseteq Z$ implies now that $G$ is as in \ref{item:cyclic}. In this case $G^{\tau}=G$, so $\cF^{\infty}_{d}\cap \cA_{d}^{G}=\{\tau\circ f \circ \tau: f\in \cF_{d}^{G}\}\cong \cF_{d}^{G}$, as claimed.
		
	It remains to consider the case where all $\phi\in\cA_{d}^{G}$ share the same $\ll_{\phi}=\{z_2=0\}$. In particular, $\ll_{\phi}=\ll_{\phi^{-1}}=\{z_2=0\}$ for every $\phi\in \cA_{d}^{G}$, so $\cA_{d}^{G}=\cF_{d}^{G}$. Since $\{z_2=0\}$ is an eigenspace of all $g\in G$, all $g\in G$ are upper-triangular, that is, $g(z_1,z_2)=(\lambda_0 z_1+\mu z_2,\lambda_1 z_1)$ for some $\lambda_0,\lambda_1\in \kk^{*}$ and $\mu\in \kk$. 
	
	We claim that $\mu=0$ for all $g\in G$. As in Lemma \ref{lem:gf}, we compute that for $q_1\in \kk^{[1]}$:
	\begin{equation*}
	e_{q_1}\circ g=\tilde{g}\circ e_{r_1},\quad \mbox{where} \quad \tilde{g}(z_1,z_2)=(\lambda_0 z_1,\lambda_1 z_2),\ 
	r_1(z_2)=\tfrac{\lambda_1}{\lambda_0}q_1(\lambda_1 z_2)+\tfrac{\mu}{\lambda_{0}}.
	\end{equation*}
	As in Lemma \ref{lem:G}, from the uniqueness of $(s,q)$ for $f_{s,q}\in \cF_{d}$, we infer that if $f_{s,q}\in \cF_{d}^{G}$ for some $q=(q_1,\dots, q_m)$ then $q_{1}=r_{1}$. Let $c\in \kk$ be the free coefficient of $q_1$. Then $q_1=r_1$ implies that
	\begin{equation}\label{eq:c}
	c(\lambda_0-\lambda_1)=\mu.
	\end{equation}

	Suppose that $\mu \neq 0$ for some $g\in G$. Then by \eqref{eq:c} $c\neq 0$, and $\lambda_{1}\neq \lambda_{0}$, so $g$ has two invariant lines, corresponding to different eigenvalues $\lambda_0$, $\lambda_1$. The first one is $\{y=0\}$, and by a linear change of coordinates, which does not change the above description, we can assume that the second one is $\{z_1=0\}$. This implies that the corresponding $\mu$ is zero, hence $c=0$ by \eqref{eq:c}, a contradiction.
	
	Therefore, $G\subseteq D$, so $G$ is one of the groups in Lemma \ref{lem:G}. If $G=G^{\tau}\de \{\tau\circ g \circ \tau:g\in G\}$ then as in the previous case we get $\tau\circ f \circ \tau\in \cA_{d}^{G}\setminus \cF_{d}^{G}$, which is impossible. Hence $G\neq G^{\tau}$, so $G$ is as in \ref{item:other}.
\end{proof}	

We now summarize the above results in the simple case when $G$ is not cyclic, treated in Theorem \ref{thm:main}\ref{item:G-non-cyclic}.

\begin{lem}\label{lem:Aut_Lie-group}
    Let $G\subseteq \Gl_{2}(\kk)$ be a non-cyclic reductive group. Then one of the following holds:
    \begin{enumerate}
        \item\label{item:not-in-k*} $\Aut^{G}(\A^{2}_{\kk})=\Gl_{2}^{G}(\kk)$
        \item\label{item:k*} $G\cong \kk^{*}$, and there is an integer $v\geq 2$ such that after some linear change of coordinates
    \begin{equation*}
        G=\{(z_1,z_2)\mapsto (\lambda^{v} z_1,\lambda z_2): \lambda \in \kk^{*}\}\ \mbox{and}\ 
        \Aut^{G}(\A^{2}_{\kk})=\{(z_1,z_2)\mapsto (\alpha_{1} z_1+\beta z_2^{v},\alpha_{2} z_2): \alpha_{1},\alpha_{2}\in \kk^{*},\beta\in \kk\}.
    \end{equation*}
    \end{enumerate}
    In particular, if $\kk=\C$ then $\Autalg^{G}(\C^{2})$ is a connected Lie group.
\end{lem}
\begin{proof}
    Because $G\subseteq \Gl_{2}(\kk)$, we have $\Aff^{G}_{2}=\Gl_{2}^{G}(\kk)$. Thus if every $G$-equivariant automorphism is affine, then \ref{item:not-in-k*} holds. Assume it is not the case, i.e.\ $\cA^{G}_{d}\neq \emptyset$ for some polydegree $d\neq \emptyset$. Because $G$ is not cyclic, Lemma \ref{lem:AG-is-Oka}\ref{item:other} implies that $\cA_{d}^{G}=\cF_{d}^{G}$. By Lemma \ref{lem:G}, after some linear change of coordinates, $G$ is a subgroup of $\kk^{*}$, acting by $(z_1,z_2)\mapsto (\lambda^{d_{1}}z_1,\lambda z_2)$. Because $G$ is not cyclic, in fact $G=\kk^{*}$. In particular, $v\de d_{1}$ is uniquely determined by the action of $G$. Lemma \ref{lem:G} implies now that $d=(v)$ and \ref{item:k*} holds.
    
    If $\kk=\C$, then in both cases $\Autalg^{G}(\C^{2})$ is a Lie group. In case \ref{item:k*}, $\Autalg^{G}(\C^{2})$ is isomorphic, as a complex manifold, to $\C^{1}\times (\C^{*})^{2}$, hence it is connected, as claimed. In case \ref{item:not-in-k*}, connectedness follows from a simple Lemma \ref{lem:Gl_connected} below.
\end{proof}

\begin{lem}\label{lem:Gl_connected}
    Let $G\subseteq \Gl_{2}(\C)$ be a linear group. Then $\Aff_{2}^{G}=\Gl_{2}^{G}(\C)$ is a connected Lie group.
\end{lem}
\begin{proof} 
    The first equality follows from the fact that $G$ acts linearly. To see that $\Gl_{2}^{G}(\C)$ is connected, note that it is a Zariski-open subset of the space of $G$-equivariant $2\times 2$ matrices, which in turn is a linear subspace of $\C^{2\times 2}$. The latter is clearly connected, hence $\Gl_{2}^{G}(\C)$ is connected, too. 
\end{proof}
	
\section{Oka properties for groups of equivariant automorphisms}\label{sec:Oka}

In this section, we prove Proposition \ref{prop:Oka}. 
Before we make it explicit in Proposition \ref{prop:Oka_explicit}, let us recall some basic notions of Oka theory. 
 For a general introduction we refer to \cite[\S 4,5]{Forstneric_book}.
\smallskip

A complex manifold $Y$ is \emph{Oka} if any holomorphic map $f\colon X_0\to Y$ from a closed subvariety $X_0\subseteq X$ of a Stein space $X$, which extends continuously to $X$, extends holomorphically, too. A fundamental theorem in Oka theory  asserts that all Oka manifolds admit various \emph{Oka properties} which put additional conditions on such extensions \cite[5.15.1]{Forstneric_book}. The one important for us is the \emph{basic Oka property with approximation and jet interpolation}, which asserts that the continuous extension $f_0$ can be deformed to the holomorphic one, say $f_1$, via a continuous family $\{f_{t}\}_{t\in [0,1]}$ of maps such that all $f_{t}|_{X_0}$ agree with $f$ up to some fixed order; and, for a fixed holomorphically convex compact $K\subseteq X$ such that $f_0|_{K}$ is holomorphic; all $f_{t}|_{K}$ are uniformly close to $f_{0}|_{K}$. A parametric version of this property holds, too: we will discuss it in Section \ref{sec:remarks}.
\smallskip

An important class of Oka manifolds are connected Lie groups, and, more generally, homogeneous spaces  \cite[5.6.1]{Forstneric_book}. Therefore, it is natural to ask for Oka properties for automorphisms groups of affine varieties. These groups are usually infinite dimensional, so the equivalence of Oka properties no longer applies. Moreover, to get reasonable results, one should restrict the class of maps in question. For example, if $X_0$ is a discrete subset of an irreducible Stein space $X$, then $f_{0}\colon X_0\to \Autalg(\C^n)$ can extend to a holomorphic map $X\to \Autalg(\C^n)$ only if it has bounded degree, see Lemma \ref{lem:coefficients}\ref{item:bounded_degree}.

Research in this direction was initiated by Forstneri\v{c} and L\'arusson \cite{FL_Oka-for-groups}, who proved that $\Authol(\C^n)$, as well as its subgroups consisting of automorphisms preserving the volume or symplectic form, satisfies the parametric Oka property with approximation and interpolation on \emph{discrete} sets $X_0$. 
\smallskip

Our Proposition \ref{prop:Oka_explicit} shows that $\Autalg^{G}(\C^2)$ admits a basic Oka property, with approximation and interpolation; for maps of bounded degree from open Riemann surfaces. Its proof is a generalization of the case $G=\{\id\}$, proved in \cite[\S 7]{FL_Oka-for-groups}. To interpolate such map at a point, loc.\ cit.\ uses the Oka property for the space $\cA_{d}$ of automorphisms of given polydegree $d$, see Section \ref{sec:FM}. By Lemma \ref{lem:AG-is-Oka}, the same holds in the $G$-equivariant setting: indeed, $\cA_{d}^{G}$ is either $\cF_{d}^{G}$, or a disjoint union of two copies of $\cF_{d}^{G}$; or a $\cF_{d}^{G}$-bundle over $\P^{1}\times \P^{1}$. Since by Remark \ref{rem:FG-is-Oka} each $\cF_{d}^{G}\cong (\C^{*})^{M}\times \C^{N}$ is Oka, $\cA_{d}^{G}$ is Oka, too; see \cite[5.6.4, 5.6.5]{Forstneric_book}. To get approximation one needs to follow this argument more carefully.

\subsection{Basic Oka property for \texorpdfstring{$\Autalg^{G}(\C^{2})$}{AutalgG}}\label{sec:Oka_proof}

The following Proposition \ref{prop:Oka_explicit} is an explicit formulation of Proposition \ref{prop:Oka}. We do not know if it holds for higher-dimensional $X$, or if it admits parametric versions, see Remark \ref{rem:dim_1} and Section \ref{sec:remarks} below. The assumption \eqref{eq:bounded_degree} is necessary by Lemma \ref{lem:coefficients}\ref{item:bounded_degree}.

\setcounter{claim}{0}
\begin{prop}\label{prop:Oka_explicit}
	Let $G\subseteq \Autalg(\C^{2})$ be a reductive subgroup. Let $X$ be a connected, open Riemann surface. Let $K$ be a compact, $\cO(X)$-convex subset of $X$ and let $X_0\subseteq X$ be a discrete set. Fix $\epsilon>0$ and a sequence $(r_{x})_{x\in X_0}\subseteq \Nn$. Let $\phi\colon \Omega\to\Autalg^{G}(\C^2)$ be a holomorphic map from some neighborhood $\Omega\subseteq X$ of $K\cup X_0$. Assume that
	\begin{equation}\label{eq:bounded_degree}
	X\ni x\mapsto \deg \phi(x) \in \Nn \quad\mbox{is bounded}.
	\end{equation}
	 Then there is a continuous map $X\times [0,1]\ni (x,t)\mapsto \phi_{t}(x)\in \Autalg^{G}(\C^2)$ such that
	\begin{enumerate}
		\item\label{item:phi_1-holo} $\phi_{0}|_{K\cup X_0}=\phi$ and $\phi_{1}\colon X\to \Autalg^{G}(\C^2)$ is holomorphic,
		\item\label{item:approximation} For every $t\in [0,1]$, $\phi_{t}$ is holomorphic on some neighborhood of $K$ and satisfies $d_{K}(\phi_{t},\phi)<\epsilon$,
		\item\label{item:interpolation} For every $t\in [0,1]$ and $x\in X_0$, $\phi_{t}$ agrees with $\phi$ up to order $r_{x}$ at $x$; that is, $\phi_{t}(x)=\phi(x)$, and, if $r_{x}>0$, $\phi_{t}$ is holomorphic at $x$ and every coefficient of $\phi-\phi_{t}$ has zero of order at least $r_{x}+1$ at $x$.
	\end{enumerate} 
\end{prop}

By Lemma \ref{lem:generic_linearization} we can choose coordinates on $\C^{2}$ in such a way that $G$ in Proposition \ref{prop:Oka_explicit} acts linearly. 
\smallskip

If $G$ is not cyclic, then by Lemma \ref{lem:Aut_Lie-group} $\Autalg^{G}(\C)$ is a connected Lie group, hence an Oka manifold \cite[5.6.1]{Forstneric_book}. Therefore, Proposition \ref{prop:Oka_explicit} holds for such $G$. The difficult part is to prove it for cyclic $G$.
%
%
\medskip

We will use the following combination of Mittag-Leffler and Weierstrass theorems, see \cite[Exercise 2.16.7 and Theorem 3.12.1]{NR_Intro_to_Riemann_Surfaces} or \cite{Florack}.
\begin{lem}\label{lem:Mittag-Leffler}
	Let $K,X_{0},\Omega \subseteq X$, $(r_{x})_{x\in X_0}$ and $\epsilon$ be as in Proposition \ref{prop:Oka_explicit}. Fix $f\in \cM(\Omega)$ which is holomorphic on $K$ and does not vanish on $K\setminus X_{0}$. Then there is a continuous map $X\times [0,1]\ni (x,t)\mapsto f_{t}(x)\in \P^{1}$ such that each $f_{t}\in \cM(\Omega)$ is holomorphic on $K$; $f-f_{t}$ has zero of order at least $r_{x}+1$ at every $x\in X_{0}$, satisfies $|f-f_{t}|<\epsilon$ on $K$, and $f_{1}\in \cM(X)\cap \cO(X\setminus X_{0})^{*}$.
\end{lem}
\begin{proof}
	For $x\in X_{0}$ let $k_{x}$ be the order of zero of $f$ at $x$ (so $k_{x}<0$ if $f$ is not holomorphic at $x$), and let $\mathcal{L}$ be the line bundle associated to the divisor $\sum_{x\in X_0}k_{x}x$. On some neighborhood $\Omega'\subseteq \Omega$ of $X_0\cup K$, $f$ trivializes $\mathcal{L}$. Since $H^{2}(X,\Z)=0$ as $X$ is a Stein space of dimension $1$, $f$ extends to a continuous trivialization $f_{0}$ of $\mathcal{L}$. The result follows from the basic Oka property for sections of $\mathcal{L}$ \cite[5.4.4]{Forstneric_book}.
\end{proof}

\begin{proof}[Proof of Proposition \ref{prop:Oka_explicit}]
	Recall that by Lemma \ref{lem:generic_linearization} we can assume that $G\subseteq \Gl_{2}(\C)$ acts linearly. If $G$ is trivial then the proposition follows from \cite[Theorem 1.2]{FL_Oka-for-groups}, so we assume that $G\neq \{\id\}$.
%
	
	\begin{claim}\label{cl:origin}
		We can assume that $\phi$ fixes the origin, i.e.\ $\phi(x)(0)=0$ for every $x\in \Omega$.
	\end{claim}
	\begin{proof}
		Because $\phi(x)$ is equivariant with respect to the linear action of $G$, for every $g\in G$ we have $g(\phi(x)(0))=g(\phi(x)(g^{-1}0))=(g\circ\phi(x)\circ g^{-1})(0)=\phi(x)(0)$, that is, $\phi(x)(0)$ is a common eigenvector of all $g\in G$, with eigenvalue $1$. Since $G\neq \{\id\}$, it follows that $\phi(x)(0)=\lambda(x)v$ for some $\lambda\in \cO(\Omega)$ and a fixed $v\in \C^{2}$. By Lemma \ref{lem:Mittag-Leffler}, there is a homotopy $\{\lambda_{t}\}_{t\in [0,1]}$ such that $\lambda_{1}\in\cO(X)$, $\lambda_{t}$ is close to $\lambda$ on $K$, and $\lambda_{t}$ agrees with $\lambda$ up to order $r_{x}$ at each $x\in X_0$. Thus to prove Proposition \ref{prop:Oka_explicit}, it suffices to prove it for a holomorphic map $\Omega \ni x\mapsto [z\mapsto \phi(x)(z)-\phi(x)(0)]\in \Autalg^{G}(\C^{2})$, which clearly fixes the origin. 
	\end{proof}

	By Claim \ref{cl:origin} we can, and therefore will, prove our proposition with $\Autalg^{G}(\C^2)$ replaced by its subgroup $\mathcal{G}$ consisting of those automorphisms which fix the origin of $\C^{2}$. More precisely, we assume that $\phi$ takes values in $\mathcal{G}$, and we will construct $\phi_{t}$ with values in $\mathcal{G}$, too.
	
	\begin{claim}
		For some neighborhood $U\subseteq \Omega$ of $K\cup X_0$, $\phi|_{U}$ extends to a continuous map $\phi_0\colon X\to \mathcal{G}$.
	\end{claim}
	\begin{proof}
		Let $U$ be an open neighborhood of $K\cup X_{0}$ such that $\bar{U}\subseteq \Omega$ and $\d U$ is smooth. Whitehead's results give a retraction $r\colon X\setminus U\to \Gamma$ onto an embedded graph $\Gamma$ containing  $\d U$: for a very detailed proof of this fact see e.g. \cite[Lemma 4]{KuL}. On the other hand, we have a retraction $d_{0}\colon \mathcal{G}\ni \alpha\mapsto \alpha'(0)\in \Gl_{2}^{G}(\C)$, see \cite[Lemma 2.1]{FL_Oka-for-groups}: indeed, since $G$ acts linearly, $\alpha'(0)$ is $G$-equivariant whenever $\alpha$ is. By Lemma \ref{lem:Gl_connected}, $\Gl_{2}^{G}(\C^2)$ is a connected manifold, so $d_{0}\circ \phi|_{\bar{U}}$ extends to a continuous map $\tilde{\phi}\colon \bar{U}\cup \Gamma \to \Gl_{2}(\C)$. Now, $\phi_{0}\de \tilde{\phi} \circ r$ is the required extension. 
		
		We remark that, since $K\cup X_0$ has no relative holes in $X$, we could have chosen $U$ in such a way that there is a retraction $h\colon \bar{U} \cup \Gamma \to \bar{U}$. This way, instead of extending $d_0\circ \phi$, one could simply put $\tilde{\phi}=d_{0}\circ \phi|_{\bar{U}}\circ h$, i.e.\ contract the graph.
	\end{proof}

	\begin{claim}\label{cl:fixed_polydegree}
	Let $U$ be a neighborhood of $K$ such that $\phi|_{U}$ is holomorphic; and let $d$ be the polydegree of $\phi|_{U}$, viewed as an element of $\Aut(\A^{2}_{\cM(U)})$. Then our proposition holds under following additional assumption: $X_0=X_{\id}\cup X_{d}$, $\phi$ agrees up to order $r_x$ with $\id$ at every $x\in X_{\id}$ and with some $\alpha_{x}\in \cA^{G}_{d}$ at every $x\in X_{d}$.
	\end{claim}
	\begin{proof}
		If $K=\emptyset$, we choose $U$ to be a neighborhood of some $x\in X_0$ such that $\phi|_{U}$ is holomorphic. Thus in any case, we can assume $U\neq \emptyset$. We can also assume that $d\neq \emptyset$. Indeed, otherwise $\phi|_{U\cup X_0}\in $ is affine, so the claim follows from the Oka property of $\Aff^{G}_{2}$. The latter is a connected Lie group by Lemma \ref{lem:Gl_connected}, hence Oka by \cite[5.6.1]{Forstneric_book}.
		
		Applying Lemma \ref{lem:AG-is-Oka} over the field $\cM(U)$, we see that $\phi|_{U}=\sigma^{-1}\circ a\circ f\circ \sigma$ for some $\sigma,a\in \Gl_{2}(\cM(U))$ and $f\de f_{s,q}\in \cF^{G}_{d}$ as in Lemma \ref{lem:G}. Here $a$ comes from \eqref{eq:section}, so it equals $\id$ in cases \ref{item:cyclic},\ref{item:other} of Lemma \ref{lem:AG-is-Oka}, and in case \ref{item:scalar} it is a composition of maps denoted in \eqref{eq:section} by $a_{\theta,\mu}$ and $a_{\eta,\lambda}$, which are of type $(x,y)\mapsto(x,\lambda_{j} x+y)$ or $(\lambda_{j} x+y,x)$ for some $\lambda_{j}\in \cM(U)$, $j\in \{1,2\}$. In cases \ref{item:cyclic},\ref{item:other} put $\lambda_{1}=\lambda_{2}=0$.
		
		Let $c\de(c_{1},\dots, c_{m})\in \cM(U)^{m}$ be a list of all nonzero coefficients of $s$, of the polynomials in $q$, and $\lambda_{j}$'s. Let $\kk=\cM(U)$ or $\cM_{X,x}$. For $\gamma\in \kk^{m}$ define $\phi^{\langle \gamma \rangle}$ by the same formula as $\phi|_{U}$, but with each $c_{j}$ replaced by $\gamma_{j}$. We claim that
		\begin{equation}\label{eq:key_claim}
		\mbox{for any } \gamma\in \kk^{m}, \phi^{\langle \gamma \rangle}\in \Aut^{G}(\A^{2}_{\kk}) \mbox{ and } \phi^{\langle \gamma \rangle}(0)=0.
		\end{equation}
		
		Clearly, $\phi^{\langle \gamma \rangle}\in \Aut(\A^{2}_{\kk})$. To see that $\phi^{\langle \gamma \rangle}(0)=0$, recall that by Lemma \ref{lem:G} $\phi(0)=0$ means that $s\in D$ in \eqref{eq:with_T}, so it is equivalent to vanishing of certain coefficient of $s$, and therefore, it holds for $\phi^{\langle \gamma \rangle}$, too. 
		
		To prove $G$-equivariance, assume first that $\gamma\in(\kk^{*})^{m}$. Then the polydegrees of $\phi^{\langle \gamma \rangle}$ and $\phi|_{U}$ are the same, so $\phi^{\langle \gamma \rangle}\in \cA_{d}$. By Lemma \ref{lem:AG-is-Oka}, modifying $a$ does not change $G$-equivariance: indeed, in cases \ref{item:cyclic},\ref{item:other} $a$ stays equal to $\id$, and in case \ref{item:scalar} any $a$ is good. Eventually, recall that by Remark \ref{rem:FG-is-Oka} $\cF_{d}^{G}\subseteq \cF_{d}$ is cut out by vanishing of certain coefficients. This vanishing holds for the ones in $\tilde{\phi}$, too, because $\tilde{\phi}$ differs from $\phi_{x}$ only by the nonzero ones. Hence $\phi^{\langle \gamma \rangle}\in \cA_{d}^{G}$. 
		
		Now, for any $\gamma\in \kk^{m}$, $\phi^{\langle \gamma \rangle}$ is a limit of $\phi^{\langle \tilde{\gamma} \rangle}$ with $\tilde{\gamma}\in (\kk^{*})^{m}$. We have proved that $\phi^{\langle\tilde{\gamma}\rangle}\in \cA^{G}_{d}$ are $G$-equivariant, hence their limit $\phi^{\langle \gamma \rangle}$ is $G$-equivariant, too. This ends the proof of  \eqref{eq:key_claim}. Alternatively, one can observe that $\phi^{\langle \gamma \rangle}$ satisfies the conditions of Lemma \ref{lem:G}, too, even though its polydegree may drop: hence the argument used above for $\gamma\in (\kk^{*})^{m}$ works for $\gamma\in \kk^{m}$,  too. 
		\smallskip
		
		Put $X_0'=\bigcup_{j=1}^{m}X'_{j}$, where $X'_{j}$ is the set of poles of $c_{j}\in \cM(U)$. By Lemma \ref{lem:perturbing_composition}, for each $x\in X_0'$ there is an integer $s_{x}>0$ such that if $\gamma_{j}$ agrees with $c_{j}$ up to order $s_{x}$ at $x$, then $\phi^{\langle \gamma \rangle}$ is holomorphic. Now for every $x\in X_0\cup X_0'$ define $v_{x}$ as $r_{x}$ if $x\in X_0\setminus X_0'$, $s_{x}$ if $x\in X_0'\setminus X_0$ and $\max\{s_{x},r_{x}\}$ if $x\in X_0\cap X_0'$.
		
		Recall that for any $x\in X_{0}$ we have $\alpha_{x}\in \cA^{G}_{d}$, so $\alpha_{x}=\phi^{\langle \eta \rangle}$ for some $\eta\in \cM_{X,x}^{m}$. It is also clear that $\id=\phi^{\langle \delta\rangle}$ for some $\delta\in \{0,1\}^{m}\subseteq \cM_{X,x}^{m}$, $x\in X_{\id}$. Thus for all $x\in X_0\cup X_0'$ there exists $\theta^{x}\in \cM_{X,x}^{m}$ such that the germ of $\phi$ at $x$ agrees with $\phi^{\langle \theta^{x} \rangle}$ up to order $v_{x}$.

		By Lemma \ref{lem:Mittag-Leffler}, for any $\epsilon_0>0$ there is a (continuous) homotopy $\{\gamma^{t}\}_{t\in [0,1]}$ such that for each $j\in \{1,\dots, m\}$, $\gamma_{j}^{0}|_{U}=c_{j}$, $\gamma^{1}_{j}\in \cM(X)\cap\cO(X\setminus X_0')$, $\gamma_{j}^{t}$ agrees with $\theta_{j}^{x}$ up to order $v_{x}$, and $|c_{j}-\gamma_{j}^{t}|<\epsilon_0$ on $K$. Put $\phi_{t}=\phi^{\langle \gamma^{t} \rangle}$. By Lemma \ref{lem:topology}, $(t,\phi)\mapsto \phi_{t}$ is continuous, and we can choose $\epsilon_0$ in such a way that $d_{K}(\phi,\phi_{t})<\epsilon$, so \ref{item:approximation} holds. Now, the choice of $v_{x}$ guarantees that \ref{item:phi_1-holo} and \ref{item:interpolation} hold, too.
	\end{proof}

	Applying Claim \ref{cl:fixed_polydegree} to $X_0\cap K$ and some smaller $\epsilon_0>0$, we get a holomorphic map $\alpha\colon X\to \mathcal{G}$ such that $d_{K}(\phi,\alpha)<\epsilon_0$ and $\alpha$ agrees with $\id$ up to order $r_{x}$ for every $x\in X_0\cap K$. The map $\beta\de \alpha^{-1}\circ\phi$ is holomorphic on $K\cup X_0$. Since $\alpha$ is holomorphic, it is of bounded degree, so $\beta$ is of bounded degree, too. Thus we can write $X_0$ as a finite disjoint union $\bigsqcup_{j=1}^{n} X_{j}$, such that for any $x\in X_{j}$, the germ of $\beta$ at $x$ 
	has polydegree $d_{j}$. Claim \ref{cl:fixed_polydegree} gives a homotopy $\{\beta^{(j)}_{t}\}_{t\in [0,1]}$ from $\beta^{(j)}_{0}=\beta$ to a holomorphic map $\beta^{(j)}_{1}\colon X\to \mathcal{G}$, such that $d_{K}(\beta,\beta^{(j)}_{1})<\epsilon_{0}$; and a holomorphic germ at $x\in X_{0}$ of $\beta^{(j)}_{t}$ agrees up to order $r_{x}$ with $\beta$ if $x\in X_{j}$, and with $\id$ if $x\in X_0\setminus X_{j}$. Therefore, at each $x\in X_0$, the map
	\begin{equation*}
	\phi_{t}\de\beta^{(n)}_{t}\circ\dots\circ\beta^{(1)}_{t}\circ \alpha.
	\end{equation*}
	agrees with $\beta\circ\alpha=\phi$ up to order $r_x$, so \ref{item:interpolation} holds. 
	Because $\beta_{1}^{(j)}$ is holomorphic, $\phi_{1}$ is holomorphic, too, so \ref{item:phi_1-holo} holds. By Lemma \ref{lem:topology}, we can choose $\epsilon_0$ in such a way that $d_{K}(\phi_{t},\phi)<\epsilon$, so \ref{item:approximation} holds as well.
\end{proof}

\begin{rem}\label{rem:dim_1}
	Although \cite[Theorem 1.2]{FL_Oka-for-groups} shows that interpolation on discrete sets works for maps from arbitrary Stein spaces; to get approximation we need to restrict our attention to maps from open Riemann surfaces. Indeed, our idea is to treat such map as an element of $\Aut(\A^{2}_{\cM(X)})$, see Lemma \ref{lem:coefficients}\ref{item:iso_of_schemes}, and decompose it according to Jung's theorem, so the factors are necessarily meromorphic. If $X$ has dimension one, we only need to control the order of their poles, which is managed by Lemma \ref{lem:perturbing_composition}. However, this approach no longer works if the local rings are not DVRs.
\end{rem}

\subsection{Remarks on the parametric Oka property}\label{sec:remarks}

In view of Proposition \ref{prop:Oka_explicit}, it is natural to ask if $\Autalg^{G}(\C^{2})$ admits any kind of \emph{parametric} Oka properties. 
\smallskip

Let us now recall this notion, see  \cite[5.4.4]{Forstneric_book} for details. Fix a closed subset $P_0$ of a topological space $P$. We say that $Y$ satisfies a \emph{parametric Oka property with interpolation} if for any subvariety $X_0$ of a Stein space $X$ and any continuous map $f_0\colon P\times X\to Y$ such that $f_0(p_0,\cdot)$ and $f_0(p,\cdot)|_{X_0}$ are holomorphic for every $p_0\in P_0$ and $p\in P$; there is a homotopy $\{f_t\}_{t\in [0,1]}$ such that $f_{1}(p,\cdot)$ is holomorphic for every $p\in P$, and $f_{t}=f_0$ on $P_0\times X \cup P\times X_0$, for every $t\in [0,1]$. As in the basic case $(P,P_0)=(\{\mathrm{pt}\},\emptyset)$, one can require $f_t$ to be uniformly close to $f_0$ on some compact set, too. 

For complex manifolds $Y$, these properties are equivalent to the basic ones, and characterize Oka manifolds \cite[5.15]{Forstneric_book}. However, in our case $Y=\Autalg^{G}(\C^2)$, the parametric version seems more subtle. 

In particular, our proof of Proposition \ref{prop:Oka_explicit} does not easily generalize to this setting. Indeed, we use an explicit description of $\Aut(\A^2_{\C})$ given by the van der Kulk theorem. In order to obtain a \emph{parametric} Oka property, one would need it for $\Aut(\A^{2}_{R})$, where $R$ is the ring of continuous functions on the parameter space. But, as we have seen in Section \ref{sec:prelim_lin}, it is very difficult to get such description when $R$ is not a field.
\smallskip

Up to now, the only parametric Oka properties for $\Autalg(\C^n)$ were obtained by Forstneri\v{c} and L\'arusson in \cite{FL_Oka-for-groups}. The main tool used for approximation (see e.g.\ in \S 3 loc.\ cit) is the  Anders\'en-Lempert theory, which builds on the fact that every polynomial vector field on $\C^n$ is a finite sum of complete algebraic vector fields. Moreover, this decomposition can be made holomorphically depending on a parameter by choosing for each homogeneous part of degree $m\geq 2$ a basis consisting of algebraic vector fields of a special kind, so called \emph{shears} and \emph{overshears}, see \cite[4.9.5]{Forstneric_book} and references there. For an algebraic proof of  this statement see \cite[Lemmas 7.5 and 7.6]{KuLo}. 

In our situation, however, there are not enough $G$-equivariant shears and overshears, as shown by the following example.
\begin{ex}
	Let $G\cong \Z_3$ be a group generated by $\zeta=\exp(2\pi i/3)$; acting on $\C^{2}$ by $\zeta\cdot (z_1,z_2)=(\zeta z_1,\zeta^{2}z_2)$. Then the $G$-equivariant shears and overshears are  $z_1^{2}p(z_1^3)\tfrac{\d}{\d z_2}$, $z_2^{2}p(z_2^3)\tfrac{\d}{\d z_1}$ and $z_2p(z_1^3)\tfrac{\d}{\d z_2}$, $z_1p(z_2^3)\tfrac{\d}{\d z_1}$ for $p\in \C^{[1]}$. Clearly, they do not span the space of all $G$-equivariant algebraic vector fields.
	
	In fact, one can show that the (non-complete) $G$-equivariant vector field $z_1^{2}z_2\tfrac{\d}{\d z_1}+z_1z_2^{2}\tfrac{\d}{\d z_2}$ does not lie in the Lie algebra generated by any complete $G$-equivariant algebraic vector fields, let alone shears and overshears. This follows from the classification of such fields, given in \cite[Theorem 5.4]{KLL}.
\end{ex}
More generally, if $G=\langle \zeta\rangle \cong \Z_d$ is a cyclic group acting on $\C^2$ by $\zeta_{d}\cdot (z_1,z_2)=(\zeta z_1,\zeta^{e}z_2)$, where $\zeta_{d}=\exp(2\pi i/d)$ and $e\in \{1,\dots,d-1\}$ is coprime to $d$, then \cite[Corollary 5.5(i)]{KLL} implies that $\C^{2}/G$ has strong algebraic density property if and only if $e|d+1$ and $e^{2}\neq d+1$. Here strong ADP for $\C^{2}/G$ means that all algebraic vector fields on $\C^{2}/G$ (equivalently, all $G$-equivariant ones on $\C^{2}$) lie in the Lie algebra generated by complete ones. 

In particular, if  $e\nmid d+1$ or $e^{2}=d+1$, e.g.\ for $(d,e)=(3,2)$, the Anders\'en-Lempert lemma does not work for $G$-equivariant vector fields; even if one allows arbitrary complete vector fields instead of just shears and overshears. This shows that the Forstneri\v{c}-L\'arusson approach cannot be applied here.

\section{Proof of Theorem \ref{thm:main}}\label{sec:proof}
	
In this section, we combine Lemma \ref{lem:KR} and Proposition \ref{prop:Oka_explicit} to remove all poles of $\psi$ in \eqref{eq:psi}.

Fix an open Riemann surface $X$, a reductive group $G$ and a holomorphic map $\nu\colon G\times X\to \Autalg(\C^{2})$ defining a holomorphic family of algebraic $G$-actions on $\C^{2}$ parametrized by $X$, see Definition \ref{def:family}. Recall that by Lemma \ref{lem:coefficients}\ref{item:iso_of_schemes}, we can identify holomorphic and meromorphic maps $X\to\Autalg(\C^{2})$ with elements of $\Aut(\A^{2}_{\cO(X)})$ and $\Aut(\A^{2}_{\cM(X)})$, respectively. Thus $\nu$ defines an action of $G$ on $\Aut(\A^{2}_{\cO(X)})$.
\smallskip

We can now list the ingredients of the proof of Theorem \ref{thm:main}.

\begin{lem}\label{lem:ingredients}
	\begin{enumerate}
		\item\label{item:psi} There exists $\psi\in\Aut(\A^{2}_{\cM(X)})$ which linearizes $G$ over $\cM(X)$, see \eqref{eq:lin_R}.
		\item\label{item:alpha} Fix $\psi$ as in \ref{item:psi}. Let $Z$ be a set of poles of $\psi$, and let $\psi_{x}\in \Aut(\A^{2}_{\cM_{X,x}})$ be a germ of $\psi$ at some $x\in Z$. Then there is a germ $\alpha_{x}\in \Aut^{G}(\A^{2}_{\cM_{X,x}})$ such that $\psi_{x}\circ \alpha_{x}$ is holomorphic at $x$.
		\item\label{item:beta} In addition to \ref{item:alpha}, fix an $\epsilon>0$, a compact, $\cO(X)$-convex subset $K\subseteq X\setminus Z$ and a subset $X_{0}\subset Z$ such that $X_0\ni x\mapsto \deg \alpha_{x}\in \Nn$ is bounded. Then there is $\beta\in \Aut^{G}(\A^2_{\cM(X)})$, holomorphic on $X\setminus X_0$, such that $\tilde{\psi}\de \psi\circ \beta$ is holomorphic on $X_0\cup X\setminus Z$ and satisfies $d_{K}(\tilde{\psi},\psi)<\epsilon$. In particular, $\tilde{\psi}$ linearizes $G$ over $\cM(X)$.
	\end{enumerate}
\end{lem}
\begin{proof}
	\ref{item:psi} follows from Lemma \ref{lem:generic_linearization} for $\kappa=\C$, $\kk=\cM(X)$.
	
	\ref{item:alpha} follows from Lemma \ref{lem:KR} for $R=\cO_{X,x}$.
	
	\ref{item:beta} By Lemma \ref{lem:perturbing_composition}, for any $x\in X_{0}$ there is an integer $r_{x}>0$ such that for any $\tilde{\alpha}_{x}$ which agrees with $\alpha_{x}$ up to order $r_{x}$, $\psi_{x}\circ\tilde{\alpha}_{x}$ is holomorphic, too. 
	
	Let $k_{x}$ be the order of the pole of $\alpha_{x}$. By Lemma \ref{lem:Mittag-Leffler}, there is $\chi\in \cO(X)$ such that $\chi$ has a zero of order $k_{x}$ at each $x\in X_{0}$, and does not vanish elsewhere. For $x\in X_0$, let $\chi_{x}$ be a germ of $\chi$ at $x$, so $\chi_{x}\cdot \alpha_{x}\in \Aut(\A^{2}_{\cO_{X,x}})$. Since $X_{0}\subseteq X\setminus K$ is discrete, we can choose a neighborhood $\Omega$ of $K\cup X_{0}$ and a holomorphic map $\phi\in \Aut(\A^{2}_{\cO(\Omega)})$ such that $\phi|_{K}=\chi|_{K}\cdot \id$ and the germ of $\phi$ at each $x\in X_{0}$ equals $\chi_{x}\cdot\alpha_{x}$. Put $M=\sup_{x\in K}\chi(x)^{-1}$.
	
	Fix $\epsilon_{0}>0$. Proposition \ref{prop:Oka_explicit} gives a holomorphic map $\phi_{1}\in \Aut^{G}(\A^{2}_{\cO(X)})$ such that $d_{K}(\phi_{1},\phi)<\tfrac{\epsilon_{0}}{M}$, and $\phi_{1}$ agrees with $\phi$ up to order $r_{x}+k_{x}$ at every $x\in X_0$. It follows that $\beta\de \chi^{-1}\cdot \phi_{1}$ satisfies $d_{K}(\beta,\id)<\epsilon_{0}$, and $\beta$ agrees with $\alpha_{x}$ up to order $r_{x}$. By definition of $r_{x}$, the latter condition means that $\psi\circ\beta$ is holomorphic at $x$. By Lemma \ref{lem:topology}, we can choose $\epsilon_{0}>0$ such that $d_{K}(\psi\circ\beta,\psi)<\epsilon$, as claimed. 
	
	The fact that $\psi\circ \beta$ linearizes $G$ follows from the $G$-equivariance of $\beta$. Indeed, for any $g\in G$ we have $(\psi\circ \beta)\circ g \circ (\psi\circ \beta)^{-1}=\psi\circ (\beta\circ g\circ \beta^{-1})\circ \psi^{-1}=\psi\circ g \circ \psi^{-1}=\rho(g)$ for some representation $\rho\colon G\to\Gl_{2}(\C)$ as in \eqref{eq:lin_R}.
\end{proof}

\begin{proof}[Proof of Theorem \ref{thm:main}] Let $\psi$ and $\alpha_{x}$, $x\in Z$, be as in Lemma \ref{lem:ingredients}. 
	
Assume first that $G$ is not cyclic.   
	By Lemma \ref{lem:Aut_Lie-group} the map $x\mapsto \deg \alpha_{x}$ is bounded (by $1$ if $G\not\subseteq \C^{*}$, and by $v$ if $G=\C^{*}$). Thus Lemma \ref{lem:ingredients}\ref{item:beta}, applied to $K=\emptyset$ and $X_0=Z$ gives $\tilde{\psi}\in \Aut(\A^{2}_{\cO(X)})$ which linearizes $G$; that is, a holomorphic map $X\to \Autalg(\C^2)$ which linearizes $G$ algebraically. 
\smallskip

Assume now that $G$ is cyclic.
	Choose an exhaustive family of $\cO(X)$-convex compact subsets  $\emptyset =K_{0}\subseteq K_{1}\subseteq \dots$ such that $K_{j-1}\subseteq \Int K_{j}$. 
	For $j\geq 1$ let $Z_{j}\de Z\cap K_{j}$ be the set of poles of $\psi$ which lie in $K_{j}$. Clearly, $Z_{j}$ is finite, and $Z=\bigcup_{j\geq 1} Z_{j}$.
	
	\begin{claim*} For every $j\geq 0$ there is a meromorphic map $\psi_{j}\colon X\to \Autalg(\C^{2})$ which linearizes $G$, such that $\psi_{j}$ is holomorphic on $Z_{j}\cup (X\setminus Z)$ and $d_{K_{j-1}}(\psi_{j},\psi_{j-1})<2^{-j}$ for $j>1$.
	\end{claim*}
	\begin{proof}
	We argue by induction on $j$. For $j=0$ we can take $\psi_{0}=\psi$. Assume that $\psi_{j-1}$ is constructed for some $j\geq 1$. Since $Z_{j}\setminus Z_{j-1}$ is finite, Lemma \ref{lem:ingredients}\ref{item:beta} applied to $X_0=Z_{j}\setminus Z_{j-1}$, $K=K_{j-1}$ and $\epsilon=2^{-j}$ gives $\psi_{j}\de \tilde{\psi}_{j-1}$ with the required properties. 
	\end{proof}	

	Now, each $K_{j}$ admits a neighborhood $\Omega_{j}$ such that $(\psi_{i}|_{\Omega_{j}})_{i=j}^{\infty}$ is a Cauchy sequence with respect to the metric $d_{K_{j}}$. Such a sequence converges uniformly on $K_{j}$ to a holomorphic map $\phi_{j}\colon \Omega_{j}'\to\Authol(\C^{2})$ defined on some (possibly smaller) neighborhood $\Omega_{j}'$ of $K_{j}$. Because $\phi_{i}|_{K_{j}}=\phi_{j}|_{K_{j}}$ for all $i\geq j$, we infer from Lemma \ref{lem:topology} that $(\phi_{j})_{j=0}^{\infty}$ converges uniformly on compacts to a holomorphic map $\phi\colon X\to \Authol(\C^{2})$. Because each $\psi_{j}$ linearizes $G$, so does their limit $\phi$.
\end{proof}

\begin{rem}\label{rem:alg_open}
	Each $\psi_{j}$ in the above proof is a map to $\Autalg(\C^2)$. However, since that space is not complete, the limit $\phi$ of $\psi_{j}$ takes values in $\Authol(\C^2)$. We do not know if one can choose $\phi$ with values in $\Autalg(\C^2)$.%
\end{rem}

To conclude, we remark that an analogue of Theorem \ref{thm:main} for holomorphic $G$-actions (i.e.\ with target of $\nu$ in Definition \ref{def:family} replaced by $\Authol(\C^n)$) is, for now, still out of reach. Indeed, we do not even know if the individual holomorphic $G$-action on $\C^2$ is linearizable. To illustrate the difficulty here, let us recall the construction \cite[Remark 4.7]{DK} of a candidate for a counterexample.  

\begin{ex}
	\label{ex:C2-nonlinearizable}
 Let $C\de \{f=0\}\subseteq \C^2$ be the image of an embedding $\C^1\into \C^2$ which is non-rectifiable, i.e.\ there is no $\alpha\in \Authol(\C^2)$ such that $\alpha(C)$ is a line. Such embeddings were constructed in \cite{FGR_non-rectifiable-C1}, cf.\ \cite[\S 2]{DK}.  Put $X= \{ (z_1,z_2,z_3) \in \C^3: f(z_1,z_2 ) = z_3^2\}$. The $\Z_2$-action $\lambda \cdot  (z_1, z_2, z_3) =  (z_1, z_2, \lambda z_3)$, $\lambda\in \{1,-1\}\cong \Z_2$ has a categorical (here = geometrical) quotient with Luna stratification $C \subset \C^2$, see \cite[\S 5.1]{Ku} for definitions. Since the embedding is non-rectifiable, this quotient cannot be biholomorphic to the Luna quotient of a linear $\Z_2$-action on $\C^2$. Therefore, if we could show that $X$ is biholomorphic to $\C^2$, the above $\Z_2$-action would give a counterexample to the holomorphic linearization problem.
\end{ex}

\bibstyle{amsalpha}
\newcommand{\etalchar}[1]{$^{#1}$}
\providecommand{\bysame}{\leavevmode\hbox to3em{\hrulefill}\thinspace}
\providecommand{\MR}{\relax\ifhmode\unskip\space\fi MR }
\providecommand{\MRhref}[2]{%
	\href{http://www.ams.org/mathscinet-getitem?mr=#1}{#2}
}
\providecommand{\href}[2]{#2}

\end{document}